\documentclass[11pt]{amsart}
\usepackage{geometry}                % See geometry.pdf to learn the layout options. There are lots.
\usepackage{graphicx}
\usepackage{amssymb}
\usepackage{epstopdf}
\usepackage{comment}
\usepackage{color}
\usepackage{pstricks, pst-node, pst-tree}
\title[Surjective polymorphisms of directed reflexive cycles]
{Surjective polymorphisms of \\ directed reflexive cycles}

\author[I. Larivi\`ere]{Isabelle Larivi\`ere}
\author[B. Larose]{Beno\^it Larose}
\author[D. E. Pazmi\~no Pullas]{David Emmanuel Pazmi\~no Pullas}

\date{\today}
\begin{document}
\newtheorem{dummy}{Dummy}[section]

\newtheorem{conj}[dummy]{Conjecture}
\newtheorem{lemma}[dummy]{Lemma}
\newtheorem{prop}[dummy]{Proposition}
\newtheorem{theorem}[dummy]{Theorem}
\newtheorem{corollary}[dummy]{Corollary}
\newtheorem{definition}[dummy]{Definition}
\newtheorem{example}[dummy]{Example}

%SHORTCUTS FOR MATHCAL

\newcommand{\cA}{\mathcal{A}}
\newcommand{\cB}{\mathcal{B}}
\newcommand{\cC}{\mathcal{C}}
\newcommand{\cD}{\mathcal{D}}
\newcommand{\cE}{\mathcal{E}}
\newcommand{\cF}{\mathcal{F}}
\newcommand{\cG}{\mathcal{G}}
\newcommand{\cH}{\mathcal{H}}
\newcommand{\cI}{\mathcal{I}}
\newcommand{\cJ}{\mathcal{J}}
\newcommand{\cK}{\mathcal{K}}
\newcommand{\cL}{\mathcal{L}}
\newcommand{\cM}{\mathcal{M}}
\newcommand{\cN}{\mathcal{N}}
\newcommand{\cO}{\mathcal{O}}
\newcommand{\cP}{\mathcal{P}}
\newcommand{\cQ}{\mathcal{Q}}
\newcommand{\cR}{\mathcal{R}}
\newcommand{\cS}{\mathcal{S}}
\newcommand{\cT}{\mathcal{T}}
\newcommand{\cU}{\mathcal{U}}
\newcommand{\cV}{\mathcal{V}}
\newcommand{\cW}{\mathcal{W}}
\newcommand{\cX}{\mathcal{X}}
\newcommand{\cY}{\mathcal{Y}}
\newcommand{\cZ}{\mathcal{Z}}

%SHORTCUTS FOR MATHBF

\newcommand{\bA}{\mathbf{A}}
\newcommand{\bB}{\mathbf{B}}
\newcommand{\bC}{\mathbf{C}}
\newcommand{\bD}{\mathbf{D}}
\newcommand{\bE}{\mathbf{E}}
\newcommand{\bF}{\mathbf{F}}
\newcommand{\bG}{\mathbf{G}}
\newcommand{\bH}{\mathbf{H}}
\newcommand{\bI}{\mathbf{I}}
\newcommand{\bJ}{\mathbf{J}}
\newcommand{\bK}{\mathbf{K}}
\newcommand{\bL}{\mathbf{L}}
\newcommand{\bM}{\mathbf{M}}
\newcommand{\bN}{\mathbf{N}}
\newcommand{\bO}{\mathbf{O}}
\newcommand{\bP}{\mathbf{P}}
\newcommand{\bQ}{\mathbf{Q}}
\newcommand{\bR}{\mathbf{R}}
\newcommand{\bS}{\mathbf{S}}
\newcommand{\bT}{\mathbf{T}}
\newcommand{\bU}{\mathbf{U}}
\newcommand{\bV}{\mathbf{V}}
\newcommand{\bW}{\mathbf{W}}
\newcommand{\bX}{\mathbf{X}}
\newcommand{\bY}{\mathbf{Y}}
\newcommand{\bZ}{\mathbf{Z}}
\newcommand{\bGS}{{\bf S}}

%SHORTCUTS FOR MATHbb

\newcommand{\bbA}{\mathbb{A}}
\newcommand{\bbB}{\mathbb{B}}
\newcommand{\bbC}{\mathbb{C}}
\newcommand{\bbD}{\mathbb{D}}
\newcommand{\bbE}{\mathbb{E}}
\newcommand{\bbF}{\mathbb{F}}
\newcommand{\bbG}{\mathbb{G}}
\newcommand{\bbH}{\mathbb{H}}
\newcommand{\bbI}{\mathbb{I}}
\newcommand{\bbJ}{\mathbb{J}}
\newcommand{\bbK}{\mathbb{K}}
\newcommand{\bbL}{\mathbb{L}}
\newcommand{\bbM}{\mathbb{M}}
\newcommand{\bbN}{\mathbb{N}}
\newcommand{\bbO}{\mathbb{O}}
\newcommand{\bbP}{\mathbb{P}}
\newcommand{\bbQ}{\mathbb{Q}}
\newcommand{\bbR}{\mathbb{R}}
\newcommand{\bbS}{\mathbb{S}}
\newcommand{\bbT}{\mathbb{T}}
\newcommand{\bbU}{\mathbb{U}}
\newcommand{\bbV}{\mathbb{V}}
\newcommand{\bbW}{\mathbb{W}}
\newcommand{\bbX}{\mathbb{X}}
\newcommand{\bbY}{\mathbb{Y}}
\newcommand{\bbZ}{\mathbb{Z}}

\newcommand{\bnote}[1]{{\color{red} [{\bf Benoit note:} #1]}}
\newcommand{\dnote}[1]{{\color{blue} [{\bf David note:} #1]}}

\newcommand{\csp}{\mathsf{CSP}}

\begin{abstract} A {\em reflexive cycle} is any reflexive digraph whose underlying undirected graph is a cycle. Call a relational structure {\em S\l upecki} if its surjective polymorphisms are all essentially unary.  We  prove that all reflexive cycles of girth at least 4 have this property. 
\end{abstract}

\maketitle
%\section{}
%\subsection{}

%\acks{We wish to thank many people for their help.}

\section{Introduction}  
Surjective polymorphisms of finite relational structures play a central role in the study of the algorithmic complexity of quantified contraint satisfaction problems (see for instance \cite{DBLP:journals/iandc/BornerBCJK09, DBLP:conf/dagstuhl/Martin17,DBLP:conf/esa/LaroseMMPSZ21}). In this paper we focus on the surjective polymorphisms of reflexive digraphs. 
These relational structures are general enough to encode a wide range of computational problems (see for instance \cite{MR1630445}), and on the other hand support a rich categorical structure: to each reflexive digraph one may naturally associate a simplicial complex which gives some information on the nature of the polymorphisms of the digraph  \cite{MR2101222}. For instance, it is known that  if the $k$-th homotopy group of the complex is non-trivial for some $k \geq 1$  but all of the proper retracts of the digraph have trivial  $k$-th homotopy group, then the digraph has only trivial idempotent polymorphisms \cite{MR2232298}. 
In particular, any digraph whose simplicial complex triangulates a sphere has this property,  among them truncated Boolean lattices, the directed 3-cycle and all cycles of girth at least 4. 

Let us call a relational structure  {\em S\l upecki} if all its surjective polymorphisms are essentially unary. Such structures are known to have Pspace-complete QCSP \cite{DBLP:journals/iandc/BornerBCJK09}; since in particular their idempotent polymorphisms are projections, if one considers the analogous problem with added unary relations, the CSP is NP-complete. 
The following question follows naturally from the previous discussion: {\em Which reflexive digraphs are S\l upecki~? } 
In a companion paper \cite{larose-pullas-preprint} we investigate this and related questions; the present paper focuses on cycles, which triangulate 1-spheres.  By reflexive cycle, we mean any reflexive digraph whose underlying undirected graph is a cycle; this allows for any orientation of edges and symmetric edges as well. Our main result is the following:

\begin{theorem} Let $n \geq 4$, and let $\bbG$ be a reflexive $n$-cycle. Then $\bbG$ is S\l upecki.  \end{theorem}

Some very special cases were known previously:  directed cycles (a special case of a result of Fearnley \cite{MR2480633}), and alternating cycles of even length (i.e. posets called {\em crowns}) (Demetrovics and Ronyai  \cite{MR1020459}).

The strategy is as follows. Let $\bbG$ be an $n$-cycle with $n \geq 4$. Suppose there exists an onto polymorphism of $\bbG$ which is not essentially unary. Then it follows easily from the relational description of the S\l upecki clone (Lemma \ref{lemma-slupecki}) that there exists an onto homomorphism from $\prod_i \bbQ_i$ onto $\bbG$ where the $\bbQ_i$ are induced subpaths of $\bbG$ of length at most $n-2$. This property motivates the following definition: we say that the $n$-cycle $\bbG$ satisfies the {\em path condition} if there exists no homomorphism from 
a product of subpaths of $\bbG$ of length at most $n-2$ onto $\bbG$ (Definition \ref{def-path-condition}); obviously  a cycle with the path condition is S\l upecki (Lemma \ref{lemma-paths}). So the rest of the paper consists of (1) characterising which cycles fail the path condition, and then (2) showing these are also S\l upecki. To achieve (1), which is the most technically involved result of this paper, we first prove that if $\bbG$ fails the path condition, then in fact there exists a homomorphism from a product of subpaths of $\bbG$ of length at most $n-2$ to $\bbG$ whose image misses an arc of $\bbG$, i.e. is onto a subpath $\bbP$ of $\bbG$ of length $n-1$ (Theorem \ref{lemma-onto-path}). This turns the problem into a purely combinatorial question on words since existence of a homomorphism to a path is  purely metric  (Theorem \ref{theorem-surj-path}). We thus obtain a description of cycles failing the path condition (Theorem \ref{theorem-char}). It follows from this characterisation that these cycles contain a self-dual path $\bbP$ of length $n-1$ that contains at least two symmetric edges. We then  invoke a  criterion from our companion paper \cite{larose-pullas-preprint} to prove that any such cycle is S\l upecki   (Theorem \ref{theorem-last}.)\\

 \noindent{\bf Acknowledgments.} We wish to thank Hubie Chen and Barnaby Martin for useful discussions. We acknowledge the help of Francis Clavette, Marc Scattolin and Jakob Thibodeau for computations.

\section{Preliminaries: notation, definitions, etc.}

\subsection{Reflexive digraphs}

A {\em digraph} $\bbG = \langle G;E \rangle$ consists of a non-empty set  $G$ of {\em vertices} and  a binary relation $E$ on $G$; the pairs in $E$ are called the {\em arcs} or {\em edges} of $\bbG$.  It is {\em reflexive} if $(x,x) \in E$ for all $x \in E$. When we consider digraphs $\bbG$, $\bbH$, etc. we denote their set of vertices by $G$, $H$, etc. We sometimes write $u \rightarrow v$ to mean that $(u,v)$ is an arc of a digraph. Consider the simple graph obtained from the digraph $\bbG$ as follows: it has the same set of vertices $G$, and two vertices $x$ and $y$ are adjacent if one of $(x,y)$ or $(y,x)$ is an arc of $\bbG$. We say $\bbG$ is {\em connected} if this  graph is connected.  We say that a digraph $\bbH$ is {\em embedded} in a digraph $\bbG$ if it is isomorphic to an induced subdigraph of $\bbG$.\\

\begin{center} {\bf In this paper, all digraphs are assumed to be reflexive.} \end{center}  \mbox{}\\

Let $\bbG$ and $\bbH$ be digraphs. A map $f:G \rightarrow H$ is a {\em homomorphism} if it preserves edges, i.e. if $(x,y)$ is an arc of $\bbG$ then $(f(x),f(y))$ is an arc of $\bbH$.  The {\em product} $\bbG \times \bbH$ of two digraphs $\bbG$ and $\bbH$ is the usual product of relational structures, i.e.  the digraph with set of vertices $G \times H$ and arcs $((g_1,h_1),(g_2,h_2))$ where $(g_1,g_2)$ and $(h_1,h_2)$ are arcs of $\bbG$ and $\bbH$ respectively; notice that $\bbG \times \bbH$ is reflexive if both $\bbG$ and $\bbH$ are reflexive. For every positive integer $k$, the product of $k$ digraphs is defined is the obvious way;  $\bbG^k$ is the product of $\bbG$ with itself $k$ times. A $k$-ary {\em polymorphism} of $\bbG$ is a homomorphism from $\bbG^k$ to $\bbG$; it is {\em idempotent} if $f(x,\dots,x)=x$ for all $x \in G$. It is {\em essentially unary} if there exists some $1 \leq i \leq k$ and a homomorphism $g:\bbG \rightarrow \bbG$ such that $f(x_1,\dots,x_k) = g(x_i)$ for all $x_j \in G$.

\begin{definition} We say that the digraph $\bbG$ is {\em S\l upecki} if all its surjective polymorphisms are essentially unary. \end{definition}

\begin{definition} Let $k \geq 0$.  A {\em path of length $k$} is a  digraph with vertex set $\{0,1,\dots,k\}$ where for each $0 \leq i \leq k-1$, one or both of the arcs $\{(i,i+1),(i+1,i)\}$ is present, and there are no other arcs.  \end{definition}

\begin{definition} Let $n \geq 3$. An  {\em $n$-cycle} is a  digraph with vertex set $\{0,1,\dots,n-1\}$ where for each $0 \leq i < n-1$, one or both of the arcs $\{(i,i+1),(i+1,i)\}$ is present,  one or both of the arcs $\{(n-1,0),(0,n-1)\}$ is present, and there are no other arcs. The integer $n$ is called the {\em girth} of the cycle.  \end{definition}

\subsection{Words}
 It will be convenient for us to view paths (and cycles) as words on the alphabet $\cA=\{+,-,*\}$, where $*$ stands for a symmetric edge. 
 As usual, if $k$ is a positive integer, a {\em word $W$ on $\cA$ of length $k$} is a finite sequence $W=x_1\cdots x_k$ where $x_i \in \cA$; the integer $k$ is the {\em length of $W$}, denoted $|W|$. There is a unique word on $\cA$ of length 0, the {\em empty word},  and we denote it by $\epsilon$. As usual, if $X=x_1\cdots x_k$ and $Y=y_1 \cdots y_l$ are words on $\cA$, the {\em concatenation} of $X$ and $Y$ is the word $XY =x_1\cdots x_ky_1 \cdots y_l$. If $X = \epsilon$ then $XY=YX=Y$. If $n$ is a positive integer and $X$ is a word, then $X^n$ is the concatenation of $n$ copies of $X$, and we declare $X^0 = \epsilon$. 
  We say that the word $X=x_1 \cdots x_k $ is a {\em subword} of the word $Y$ if there exist words $W_0,\dots,W_k$ such that $Y =W_0 x_1 W_1 x_2 \dots x_k W_k$, i.e. $X$ is obtained from $Y$  by removing some symbols. The empty word is a subword of every word. The word $X$ is a {\em substring} of the word $Y$ if there exist words $A$ and $B$ such that 
    $Y = AXB$; if $A $ is empty, $X$ is a {\em prefix} of $Y$, and if $B$ is empty, $X$ is a {\em suffix} of $Y$. {\em In the rest of this paper, unless otherwise stated, all words are on $\cA$.}  \\

 The next definition interprets words as paths and vice-versa.

  \begin{definition} Let $k$ be a positive integer. 
  
  \begin{enumerate}
 \item Let  $W=x_1 \cdots x_k$ be a non-empty word on $\cA$. The  {\em path associated to $W$}, denoted $P(W)$, is the digraph with vertex set  $\{0,\dots,k\}$ where  $(i,i+1)$ is an edge if $x_i = +$,   $(i+1,i)$ is an edge if $x_i = -$, and both $(i,i+1)$ and  $(i+1,i)$ are edges if $x_i = *$.  We let $P(\epsilon)$ be the one-element path (with vertex set $\{0\}$). 
 
 \item Given a path $\bbP$ of length $k$, define the word $w(\bbP) =x_1x_2\cdots x_k $ as follows: $x_i = *$ if both $(i,i+1)$ and  $(i+1,i)$ are edges of $\bbP$; otherwise, $x_i = +$ if $(i,i+1)$ is an edge of $\bbP$ and $x_i = -$ if $(i+1,i)$ is an edge of $\bbP$. If $\bbP$ has length 0 then $w(\bbP) = \epsilon$. 
 
  \end{enumerate}\end{definition}
  
  \begin{figure}[htb]
\begin{center}
\includegraphics[scale=0.7]{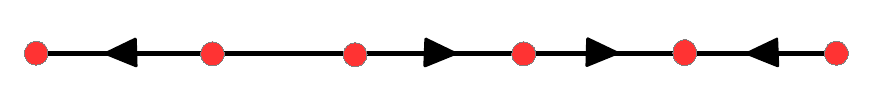}
 \caption{The path $\bbP = P(W)$ where $W = -*++-$.} \label{cycle-even-odd}
 \end{center}
\end{figure}

We now define an ordering and an involution on the set of words on $\cA$. 
 
  \begin{definition}  Let $U$ and $V$ be words of length $k$ and $l$ respectively. 
  
 We say that a homomorphism $f:P(U) \rightarrow P(V)$ is {\em end-point preserving} if $f(0)=0$ and $f(k)=l$;  and we write $U \leq V$ if there exists an end-point preserving homomorphism from $P(V)$ onto $P(U)$.

 \end{definition}

  It is a simple exercise to verify that this defines a partial ordering on the set of words, i.e. it is reflexive, antisymmetric and transitive. 
 In this ordering, the empty word $\epsilon$ is a minimal element, it is covered by $*$, which lies below $+$ and $-$ which are incomparable.

    \begin{definition} 
  
 Let $W$ be a word. If $W = \epsilon$ then $\overline{W}=\epsilon$; otherwise, if $W=x_1\cdots x_k$ let $\overline{W} = \overline{x_k}\cdots \overline{x_1}$ where $\overline{+}$ is $-$, $\overline{-}$ is $+$ and  $\overline{*}$ is $*$.

  \end{definition}

Intuitively, the involution of a word $W$ is the word $w(P') $ where $P'$ is the path $P(W)$ traversed backwards. For example, $\overline{+*-+-}$ is $+-+*-$. We shall say that a path $\bbP$ is {\em self-dual} if $w(\bbP) = \overline{w(\bbP)}$.

%It will be convenient for us to use the following notation: 

\begin{definition} Let $\bbG$ be a digraph, let  $x$ and $y$ be vertices of $\bbG$ and let $W$ be a word of length $k$. We write $x\, W y$ if there exists a homomorphism of $P(W)$ to $\bbG$ mapping 0 to $x$ and $k$ to $y$.   \end{definition}

  It will also be convenient for us to denote cycles by words on $\cA$: the word $W$ lists the succession of arcs of the cycle starting from some fixed vertex.

  \begin{definition}
 Let  $W=x_1\cdots x_k$ be a non-empty word. The   {\em cycle associated to $W$}, denoted $C(W)$, is the digraph with vertex set  $\{0,\dots,k-1\}$ where  for each $1 \leq i \leq k$,
   $(i-1,i)$ is an edge if $x_i = +$,   $(i+1,i)$ is an edge if $x_i = -$, and both $(i,i+1)$ and  $(i+1,i)$ are edges if $x_i = *$, where the indices are considered modulo $k$.
   
\end{definition}

\begin{figure}[htb]
\begin{center}
\includegraphics[scale=0.4]{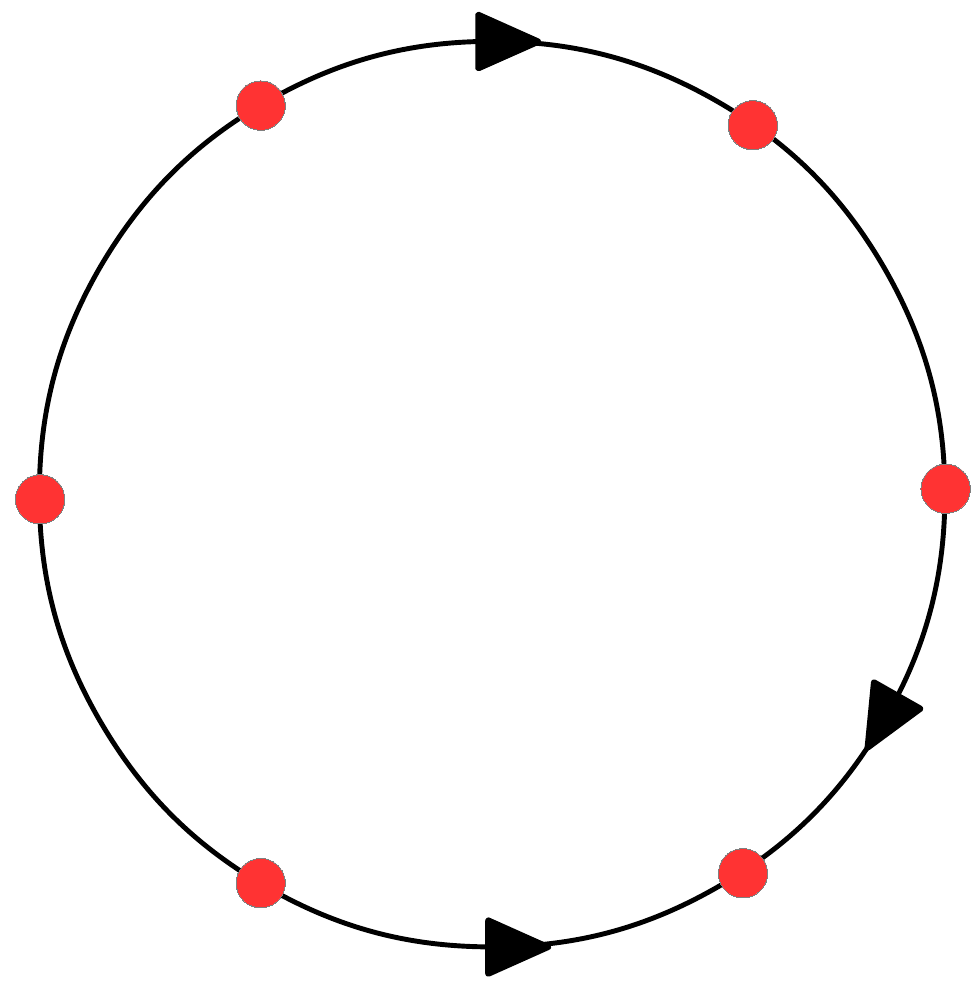}
 \caption{The cycle $\bbG = C(W)$ where $W = +*+-**$.} \label{cycle-even-odd}
 \end{center}
\end{figure}

Notice that, {\em up to isomorphism}, there are several different words that  represent a given cycle, depending on the chosen starting point and the orientation. We shall feel free to write $\bbG = C(W)$ for any of these words. 

Furthermore we shall  often  abuse notation;  for instance, we may write $\bbP \leq \bbQ$ where $\bbP$ and $\bbQ$ are paths, meaning of course that $w(\bbP) \leq w(\bbQ)$; or 
if $\bbG$ is a cycle, $\bbS$ is a path,  we may write $\bbG = (\bbS**)^{2}\, \bbS+$ instead of $\bbG = C((w(\bbS)**)^2 w(\bbS)+)$, and so on. \\

The next lemma gathers some useful consequences of the definitions.

\begin{lemma} \label{lemma-utilities} Let $X,Y,W$ be words and let $a,b \in \cA$.
\begin{enumerate}

\item Let $\bbP$ and $\bbQ$ be paths. If there exists an end-point preserving map from $\bbP$ onto $\bbQ$ then there exists one which is monotone, i.e. a map $f$ satisfying $f(i) \leq f(j)$ whenever $i \leq j$. 
\item If $\bbP \geq \bbQ$, and $\bbP=aX$ and $\bbQ=bY$ where $a \geq b$ then there exists an end-point preserving monotone map from $\bbP$ onto $\bbQ$ that sends $a$ onto $b$. 
\item If $X$ is a subword of $Y$ then $X \leq Y$; if $X \leq Y$ then $|X| \leq |Y|$. 
\item $X \leq Y$ if and only if $\overline{X} \leq \overline{Y}$; $\overline{XY}=\overline{Y}\,\overline{X}$; 
\item  if $aX \geq bY$ then $X \geq Y$; and if $a \ngeq b$ then $X \geq bY$. 
\item let $f:\bbG \rightarrow \bbH$ be a homomorphism and let $x,y \in G$. If $x\, W y$ then $f(x)\, W  f(y)$. 
\item if $X \leq A$ and $Y \leq B$ then $XY \leq AB$. 
\end{enumerate}
\end{lemma}
 
 \begin{proof} 
 (1) We use induction on $|\bbP|+|\bbQ|$: the result is trivial if $|\bbP|+|\bbQ|=0$. Let $f$ be an end-point preserving homomorphism from $\bbP$ onto $\bbQ$. (i) Suppose first that  $0$ is the only vertex of $\bbP$ mapped to $0$ by $f$, then $f(1) = 1$. Let $\bbP'$ is the subpath of $\bbP$ induced by the vertices other than 0; then by induction hypothesis we have a monotone map from $\bbP'$ onto the subpath of $\bbQ$ induced by the vertices other than 0, and we can extend it to a monotone map from $\bbP$ onto $\bbQ$. (ii) Otherwise, there exists an $x >0$ such that $f(x) = 0$; let $\bbP'$ be the subpath of $\bbP$ induced by $\{i: i \geq x\}$.  By induction hypothesis there exists a monotone map $g$ from $\bbP'$ onto $\bbQ$. Setting $h(t) = 0$ if $t < x$ and $h(t) = g(t)$ if $t \geq x$ defines a monotone homomorphism from $\bbP$ onto $\bbQ$. 
 
 (2) By (1) there exists a monotone map $f$ from $\bbP$ onto $\bbQ$. Let $x > 0$ such that $f(x) = 1$. Define a new map $g$ as follows: $g(0)=0$, $g(t)=1$ if $1 \leq t \leq x$ and $g(t)=f(t)$ for $t \geq x$. Since $a \geq b$, $g$ is a homomorphism, and it is clearly monotone and end-point preserving.
 
 Statements (3)-(7) are also easy.

 \end{proof}

\section{The path condition}

In this section, we  introduce a property that guarantees a cycle is S\l upecki, and prove a combinatorial characterisation that will allow us, in the next section,   to describe  explicitly those cycles that fail this property.

  \begin{definition} \label{def-path-condition} We say that the $n$-cycle $\bbG$ satisfies the {\em path condition} if one cannot find 
    paths $\bbQ_i$ of length $n-2$ that embed in $\bbG$ and a surjective homomorphism $f:\prod_i \bbQ_i \twoheadrightarrow \bbG$.
       \end{definition}

 Let $\theta \subseteq G^n$ be an $n$-ary relation on the non-empty set $A$, and let $f:A^k \rightarrow A$ be a $k$-ary operation on $A$. We say that $f$ {\em preserves} $\theta$ if, whenever $(x^1_j,\dots,x^n_j) \in \theta$ for $j=1,\dots,k$, we have that $(f(x^1_1,\dots,x^1_k),\dots,f(x^n_1,\dots,x^n_k)) \in \theta$.  The following fact can be found in \cite{MR2254622}. 
 
 \begin{lemma}[\cite{MR2254622}, Lemma 4.1] \label{lemma-slupecki} Let $A$ be an $n$-element set, $n \geq 2$. Then the operations on $A$ preserving the relation
 $$\theta = \{(a_1,\dots,a_n): |\{a_1,\dots,a_n\}| < n \}$$
 are precisely the essentially unary operations and the non-surjective operations on $A$.
 
 %\footnote{The clone $Pol(\theta)$ is the {\em S\l upecki} clone; it is known to be a maximal clone, and it is quite easy to see that it {\em contains} all non-surjective operations and all unary operations.} 
 \end{lemma}

\begin{comment}
In this section we consider the problem of showing that cycles (of girth at least 4) are S\l upecki; we shall show with our technique that this holds for symmetric cycles, and cycles with no consecutive symmetric edges. Here is the strategy: by Lemma \ref{lemma-slupecki}, it suffices to show that every polymorphism of the cycle $\bbG$ preserves the S\l upecki relation. Working by contradiction, suppose there exists a counter-example $f$ of arity $k$; then $f$ is surjective (of course), and maps $k$ ``columns'' from the S\l upecki relation onto a column whose entries are all distinct. Since for a cycle, every proper subset is contained in a path, this means that the restriction of $f$ to the product of $k$  subpaths of $\bbG$ maps onto $\bbG$. If one can get a handle on such homomorphisms, perhaps it is possible to show that none can ``extend'' to a full polymorphism from $\bbG^k$ to $\bbG$. For instance, in the case of a symmetric cycle, or a cycle with no consecutive symmetric edges, it turns out that there are {\em no} homomorphisms from a product of subpaths onto $\bbG$. Preliminary computations hint that this behaviour might hold for other cycles, (except for the $n$-cycle with exactly $n-1$ symmetric edges.)
\end{comment}

 \begin{theorem} \label{lemma-paths} If a cycle satisfies the path condition then it is S\l upecki. \end{theorem}
 
 \begin{proof} Suppose that the cycle $\bbG$ is not S\l upecki; by the previous lemma,  there exists a polymorphism $f:\bbG^k \rightarrow \bbG$  that does not preserve the S\l upecki relation $\theta$. Then there exist $k$ tuples in $\theta$ that $f$ maps to some tuple outside $\theta$. 
 For each tuple in $\theta$, there exists a proper subset of $G$ containing its coordinates; taking a largest one, it is a path of length $n-2$ (obtained from $\bbG$ by removing a single vertex.)  Since $f$ maps the $k$ columns to a tuple whose coordinates are all vertices of $\bbG$, in particular, the restriction of $f$ to the product of the $k$ paths is onto $\bbG$. 
 \end{proof}

Our next goal is to show that, if an $n$-cycle $\bbG$ does not satisfy the path condition, then there exists a homomorphism of a product of subpaths of $\bbG$ of length $n-2$ onto a subpath of $\bbG$ of length $n-1$ (Lemma \ref{lemma-onto-path}). This will allow us to reformulate the path condition is a purely combinatorial way (Theorem \ref{theorem-surj-path}). 

We require some basic facts about discrete homotopy of reflexive digraphs;  all details can be found in  \cite{MR2101222}, see also \cite{MR2232298}.  The {\em one-way infinite fence} $\bbF$ is the following (infinite) reflexive digraph: its base set is the set of non-negative integers, and $(n,m)$ is an arc if $ n = m$ or $n$ is even and $|n-m| = 1$. Let $\bbG$ be a digraph and $u,v \in G$. We say that $p$ is  {\em a path in $\bbG$ from $u$ to $v$} if $p$ a homomorphism from $\bbF$ to $\bbG$ such that there exists an $N \geq 0$ with $f(0)=u$ and $f(t)=v$ for all $t \geq N$. This corresponds in a straightforward way to the intuitive notion of a path in $\bbG$ from $u$ to $v$\footnote{These paths in $\bbG$ are {\em not} directed paths, but rather sequences of vertices that are paths in the underlying undirected reflexive graph.}; in particular, if $f$ is path in $\bbG$ from $u$ to $v$ and $g$ is a path in $\bbG$ from $v$ to $w$, then we can define a path $fg$ from $u$ to $w$ in the natural way, and we can also define a path $p^{-1}$ in $\bbG$ from $v$ to $u$ (see  \cite{MR2101222}). If $p$ is a path in $\bbG$ from $u$ to $v$, and $f:\bbG \rightarrow \bbH$ is a homomorphism, then clearly the composition of $f$ with $p$ is path in $\bbH$ from $f(u)$ to $f(v)$, which we denote by $f(p)$.  It is immediate to verify that $f(pq) = f(p)f(q)$,  $f(p^{-1})=f(p)^{-1}$ and $(pq)^{-1} = q^{-1}p^{-1}$. 

Let $a \in \bbG$. Then  $\bbF(\bbG,a)$ denotes the (infinite) digraph whose vertices are all paths in $\bbG$ from $a$ to $a$,  where $(f,g)$  is an arc if, whenever $(s,t)$ is an arc of $\bbF$ then $(f(s),g(t))$ is an arc of $\bbG$.  Two paths in $\bbF(\bbG,a)$ are {\em homotopic} if they are in the same connected component. It turns out that the composition of paths induces a group structure on the connected components of $\bbF(\bbG,a)$, and we denote this group by $\sigma_1(\bbG,a)$. If $\bbG$ is connected, the choice of $a$ is immaterial; and it can be shown that $\sigma_1(\bbG)$ is isomorphic to the fundamental group of the triangulation of the following simplicial complex associated to $\bbG$: its vertices are the vertices of $\bbG$, and its simplices are  the subsets of $G$ that are the homomorphic image of some totally ordered set (i.e. a chain) in $\bbG$. Notice in particular that $\sigma_1(\bbG)$ is trivial whenever $\bbG$ is a product of paths. %, and is isomorphic to $\bbZ$ if $\bbG$ is a cycle of girth at least 4. 

Our strategy to prove Lemma \ref{lemma-onto-path} is simple. Given an $n$-cycle $\bbG$,  we construct a (two-way infinite) path  $\widehat{\bbG}$ which admits a homomorphism $\pi$ onto $\bbG$, such that, for any homomorphism $f:\bbK \rightarrow \bbG$ where $\bbK$ is connected and has trivial fundamental group,  there exists a homomorphism $\widehat{f}:\bbK \rightarrow \widehat{\bbG}$ such that $\pi \circ \widehat{f} = f$: the path $\widehat{\bbG}$ is the natural analog of the universal covering space of the cycle $\bbG$. The homomorphism $f$ will be onto precisely when the image of $\widehat{f}$ has size at least $n-1$; hence, there exists a homomorphism from $\bbK$ onto $\bbG$ if and only if there is one from $\bbK$ onto some subpath of $\bbG$ of length $n-1$. \\ %(since the path $\widehat{\bbG}$ can be mapped onto that subpath). 

 Let $\bbG$ be an $n$-cycle and let $a \in G$.  Let $p$ be a path from $u$ to $v$ in $\bbG$.  Define $\alpha(p)$ to be the number of pairs $(t,t+1)$ in $\bbF$ such that $(f(t),f(t+1))=(a-1,a)$ minus the number of pairs $(t,t+1)$ in $\bbF$ such that $(f(t),f(t+1))=(a,a-1)$. Intuitively, this corresponds to  the ``winding number'' of the path $p$.

 Our first goal is to show that $\alpha$ is invariant under homotopy.
To prove this, we will use the notion of non-refinable edge from \cite{DBLP:journals/dm/MarotiZ12}. We simply rephrase Mar\'oti and Z\'adori's Lemma 2.10 for our purposes (the proof they give is still valid in the more general context of $Hom(\bbA,\bbG)$ even though they formulate it only in the case of $Hom(\bbG,\bbG)$; see section \ref{section-five} for a definition of these digraphs.)
 
Let $(p,q)$ be an edge of $\bbF(\bbG,a)$.  Let $Neq(p,q)$ denote the (finite) set of all $t \in F$ such that $p(t) \neq q(t)$. The edge $(p,q)$ is {\em non-refinable} if there is no non-empty proper subset $A$ of $Neq(p,q)$ such that changing the value of $q$ to that of $p$ on $A$ yields an element of $\bbF(\bbG,a)$. 
 Call an edge $(u,v)$ of $\bbF$ {\em critical for $(p,q)$} if $(q(u),p(v))$ is not an edge of $\bbG$. 
 
 \begin{lemma}[Based on Lemma 2.10 of  \cite{DBLP:journals/dm/MarotiZ12}] Let $\bbG$ be a digraph and let $a \in G$. \begin{enumerate}
 \item If $p,q \in \bbF(\bbG,a)$ are homotopic then there exists a path of non-refinable edges from $p$ to $q$;
 \item An edge $(p,q)$ is non-refinable if and only if $|Neq(p,q)| \leq 1$. 
 \end{enumerate}
 \end{lemma}
 
 \begin{proof}  (1) It clearly suffices to prove the result when $p\rightarrow q$. If $(p,q)$ is non-refinable, we are done. Otherwise, there exists $r \in  \bbF(\bbG,a)$ obtained by changing the value of $q$ on certain, but not all, values in $Neq(p,q)$ to that of $p$. Notice in particular that $r(t) \in \{p(t),q(t)\}$ for all $t$ and hence  $p \rightarrow r \rightarrow q$; secondly, it is clear that $Neq(r,p)$ and $Neq(r,q)$ are proper non-empty subsets  of $Neq(p,q)$. By repeating the argument, this process must terminate and we obtain the desired path. \\
 
 (2)  Necessity is immediate. Now suppose that $|Neq(p,q)| >1$. Consider the digraph $\bbD$ whose set of vertices is $V=Neq(p,q)$ and whose edges are the edges critical for $(p,q)$. Since this digraph is not strongly connected (it is a subgraph of $\bbF$ which is acyclic), we can find  a non-empty, non-trivial strong component $A$ of $\bbD$ such that no critical edge goes from $V \setminus A$ to $A$; we can thus change the value of $q$ on $A$ to that of $p$ to obtain an $r \in \bbF(\bbG,a)$; $r$ is a homomorphism by definition of critical edge and choice of $A$, and $r$ satisfies the ``boundary'' requirements of $\bbF(\bbG,a)$ because  $0 \not\in Neq(p,q)$  so $r(0)=a$; and $p(t)=q(t)=a$  for all $t \geq N$ for some $N$, so the same holds for $r$.     \end{proof}
 
  \begin{lemma} \label{lemma-invariant}Let $p,q  \in \bbF(\bbG,a)$. If $p$ and $q$ are homotopic then $\alpha(p) = \alpha(q)$. \end{lemma}
  
  \begin{proof} By the last lemma, it suffices to prove that, if $p \rightarrow q$ in  $\bbF(\bbG,a)$ and $p$ and $q$ differ in exactly one place, then $\alpha(p) = \alpha(q)$.  Let $\{c,d\}=\{a-1,a\}$. The result is clear if the edges of $\bbF$ mapped by $p$ and $q$ onto $\{c,d\}$ are the same. 
  Hence we may assume there exists a unique edge $e$ of $\bbF$ that one of $p,q$ maps onto $\{c,d\}$ but not the other.  Let $e=(u,v)$ and suppose that $(p(u),p(v))=(c,d)$ (all the other cases are identical.)   Then $p(u) \rightarrow q(v)$ and $p(v)\rightarrow q(v)$ implies that $q(v) \in \{c,d\}$ (keeping in mind that $\bbG$ is a cycle).
 Suppose first that  $q(v) = c$. Then $q(u)=p(u)=c$. By definition of $\bbF$, there exists $w$ a neighbour of $v$ distinct from $u$; by the same transitivity argument as above we get that $p(w) \in \{c,d\}$, and of course $q(w)=p(w)$. If $p(w)=d$ then both $p$ and $q$ traverse $(c,d)$ exactly once on $\{u,v,w\}$, and if $p(w)=c$ then $p$ traverses both $(c,d)$ and $(d,c)$ once while $q$ traverses neither; in both cases we get that $\alpha(p)=\alpha(q)$. Suppose now that  $q(v)=d$. Then $c=p(u) \rightarrow q(u) \rightarrow q(v)=d$ implies that  $q(u) \in \{c,d\}$, hence by hypothesis $q(u)=d$. Repeating the argument above using the neighbour $w'$ of $u$ distinct from $v$ (notice that $u \neq 0$ since one of $p(u)$ or $q(u)$ is different from $a$)  will complete the proof. \end{proof}

% Let $\bbG$ be an $n$-cycle; we define a new digraph, which is a two-way infinite path, by ``unravelling'' $\bbG$, i.e.
 
 \begin{definition}   Let $\bbG$ be an $n$-cycle, and let $a \in G$. The {\em universal covering space $\widehat{\bbG}$ of $\bbG$} is the  (infinite) reflexive digraph with vertex set $G \times \bbZ$, and whose arcs are all pairs $((x,i),(y,i))$ such that $(x,y) \in E(\bbG)$ and $\{x,y\} \neq \{a-1,a\}$, together with the arcs $((a-1,i),(a,i+1))$ if $(a-1,a) \in E(\bbG)$ and  $((a,i+1),(a-1,i))$ if $(a,a-1) \in E(\bbG)$, and this for all $i \in \bbZ$.  \end{definition}
 
 Notice that, in the previous definition, the choice of $a$ is immaterial, i.e. the resulting digraphs will all be isomorphic. We chose this way of defining the universal covering space to make the following proof easier.
 It is clear by construction that the map $\pi:\widehat{\bbG} \rightarrow \bbG$ defined by $\pi(x,i)=x$ is a homomorphism onto $\bbG$.

 \begin{lemma} \label{lemma-lift} Let $\bbK$ be a  connected digraph with $\sigma_1(\bbK) = 0$, and let $f:\bbK\rightarrow \bbG$. Then there exists a map $\widehat{f}:\bbK \rightarrow \widehat{\bbG}$ such that $\pi \circ \widehat{f} = f$. \end{lemma}
 
 \begin{proof}
 Choose a vertex $u \in K$ and let $f(u)=a$.  Let $v \in K$.     \\
 
 \noindent{\bf Claim.} {\em If $p$ and $q$ are any paths from $u$ to $v$ in $\bbK$, then $\alpha(f(p)) = \alpha(f(q))$. }  \\
 
 \noindent{\em Proof of Claim.} Consider the discrete homotopy groups $\sigma_1(\bbK,u)$ and $\sigma_1(\bbG,a)$; $f$ induces a group homomorphism from the first to the second. 
 Since $\sigma_1(\bbK,u)$ is trivial, for any two paths $p,q$ from $u$ to $v$ we have that $pq^{-1}$ is homotopic to the constant path at $u$. Applying $f$, we obtain that $f(pq^{-1})$ is homotopic to the constant path at $a$; in particular, since $\alpha$ is preserved under homotopy by Lemma \ref{lemma-invariant}, $0 = \alpha(f(pq^{-1})) = \alpha(f(p)f(q^{-1}))=\alpha(f(p)f(q)^{-1})$. 
 Since clearly $\alpha(f(p)f(q)^{-1})=\alpha(f(p)) - \alpha(f(q))$ the claim follows.\\
 
 By the claim, if we set $\widehat{f}(v) = (f(v), \alpha(f(p)))$ for any path $p$ from $u$ to $v$, this function is well-defined; clearly $\pi \circ \widehat{f} = f$, so it remains to prove $\widehat{f}$ is a homomorphism.  This is in fact straightforward: let $v \rightarrow w$ be an arc in $\bbK$. Let $p$ be any path from $u$ to $v$, and let $p'$ be the path from $u$ to $w$ which is $p$ followed by the arc $(v,w)$. Let $\alpha = \alpha(f(p))$ and let $\alpha' = \alpha(f(p'))$. Then $\widehat{f}(v) = (f(v),\alpha)$ and $\widehat{f}(w) = (f(w),\alpha')$. Notice that $\alpha \neq \alpha'$ precisely when $\{f(v),f(w)\}=\{a-1,a\}$. Thus if $\alpha = \alpha'$ we are done. Otherwise, we have that  either $(f(v),f(w)) = (a-1,a)$ and then $\alpha' = \alpha +1$ or  $(f(v),f(w)) = (a,a-1)$ and then $\alpha' = \alpha -1$; in either case $(\widehat{f}(v),\widehat{f}(w))$ is an arc. 
 
 \end{proof}
 
 The next result shows that the path condition  can be reformulated as a problem of determining if there exists a homomorphism from products of subpaths of $\bbG$ onto a subpath of $\bbG$ rather than $\bbG$ itself;  we then describe a simple combinatorial criterion to determine if such a homomorphism exists (Theorem \ref{theorem-surj-path}).

 %\bnote{I think it might be useful to reformulate the next result in terms of a single operation at a time, i.e. just like the proof (1) implies (2) actually does. In fact, we might generalise the statement, the paths Pi %can be any paths I think. }
 
 \begin{lemma} \label{lemma-onto-path} Let $\bbG$ be an $n$-cycle. Then the following are equivalent:
  \begin{enumerate}
     \item  $\bbG$ fails the path condition;
     \item there exist paths $\bbQ_i$ of length $n-2$ that embed in $\bbG$, a path $\bbP$ of length $n-1$ that embeds in $\bbG$ and a surjective homomorphism $f:\prod_i \bbQ_i \rightarrow \bbP$. 
     \end{enumerate}
\end{lemma}
 
 \begin{proof}
 $(2) \Rightarrow (1)$ is immediate. $(1) \Rightarrow (2)$: suppose there exists a surjective homomorphism $f:\bbK\rightarrow \bbG$ where $\bbK$ is a product of subpaths of $\bbG$ of length $n-2$. By Lemma \ref{lemma-lift}, we have a homomorphism $\widehat{f}:\bbK \rightarrow \widehat{\bbG}$ such that $\pi \circ \widehat{f} = f$. In particular, the image of $\widehat{f}$ is connected, and hence is a path, and has length at least $n-1$. It now suffices to observe that any path can be retracted to a subpath sending the endpoints of the longer path to the endpoints of the shorter one; the path of length $n-1$ is a subpath of $\bbG$ by construction of $\widehat{\bbG}$. 
 \end{proof}

 \begin{theorem} \label{theorem-surj-path} Let $\bbG$ be an $n$-cycle, and let $\bbQ_1,\dots,\bbQ_l$ denote its induced subpaths of length $n-2$.  The following are equivalent:
 \begin{enumerate}
 \item $\bbG$ satisfies the path condition;
 \item for every induced subpath $\bbP$ of $\bbG$ of length $n-1$, there exists a word $W$  such that $W \geq \bbQ_i$ for all $i$ but $W \not\geq \bbP$.

 \end{enumerate} \end{theorem}
 
Before we launch into the proof, we shall require some facts about extending partial maps to paths. The next result is a consequence of Theorem 42 in \cite{MR2854146}, see also Theorems 19 and 20 in \cite{MR3631056} as well as the comment that follows them. 

\begin{lemma} \label{lemma-extend} Let $\bbP$ be a path. Let $\bbK$ be a digraph, let $X \subseteq K$ and let $g:X \rightarrow P$ be any map. There exists a homomorphism from $\bbK$ to $\bbP$ whose restriction to $X$ is $g$ if and only if, for every $x,y \in X$, and any word $W$, if $x W y$ in $\bbK$ then $g(x) W g(y)$ in $\bbP$. \end{lemma}

%\bnote{I need to formalise this a bit more, using coloured digraphs and homomorphisms between them. Later}
\begin{proof}  Necessity is immediate by Lemma \ref{lemma-utilities} (6).  Now suppose there exists no homomorphism from $\bbK$ to $\bbP$ whose restriction to $X$ is $g$. 
Consider the relational structure $\bbP^c$ consisting of the path $\bbP$ together with all singleton unary relations $\{p\}$, $p \in P$; let $\tau$ denote the signature of this structure.   Consider the $\tau$-structure $\bbK^c$ consisting of the digraph $\bbK$ where each element $x \in X$ is constrained by the unary relation $\{g(x)\}$. Then $\bbK^c$ does not admit a homomorphism to $\bbP^c$. It follows from the results mentioned above that $\bbP^c$ has path duality, hence there exists some $\tau$-path $\bbQ^c$ that admits a homomorphism $h$ to $\bbK^c$ but no homomorphism to $\bbP^c$. Because $\bbP$ is reflexive, it is easy to see that such a $\tau$-path consists of a directed path with endpoints constrained by some unary relations; let's denote the path by $\bbQ$, its endpoints by $u$ and $v$ which are constrained by the unary relations $\{a\}$ and $\{b\}$ respectively; let $h:\bbQ^c \rightarrow \bbK^c$ be the homomorphism. Since $h$ is a $\tau$-homomorphism we must have that $x=h(u)$ and $y=h(v)$ belong to $X$ and $g(x)=a$, $g(y)=b$.   If $W$ denotes the word $w(\bbQ)$ where $\bbQ$ is the underlying path of $\bbQ^c$, then we have that  $x W y$;  but $g(x)  W g(y)$ does not hold, as desired. \end{proof}

 \begin{proof} (Theorem \ref{theorem-surj-path})  (1) $\Rightarrow$ (2): suppose that there exists an induced subpath $\bbP$  of $\bbG$ of length $n-1$ such that, for every word $W \geq \bbQ_i$ for all $1 \leq i \leq l$ we have $W \geq \bbP$. For each $i$ let $a_i,b_i$ denote the endpoints of $\bbQ_i$, let $X = \{\alpha,\beta\} \subseteq \prod_i \bbQ_i$ where $\alpha = (a_1,\dots,a_l)$ and let $\beta=(b_1,\dots,b_l)$; and let $g:X \rightarrow \bbP$ be defined by $g(\alpha)=a$ and $g(\beta)=b$ where $a$ and $b$ are the endpoints of $\bbP$. Let $W$ be a word such that $\alpha W \beta$ in $\prod_i \bbQ_i$; composing with the projection maps this implies that $W \geq \bbQ_i$ for all $i$ and so $W \geq \bbP$ by hypothesis; this means that $a W b$ in $\bbP$, and the last lemma implies that there exists a homomorphism from $\prod_i \bbQ_i$ to $\bbP$ that extends $g$. Since $\prod_i \bbQ_i$ is connected, so is its image which contains the endpoints of $\bbP$, so this map is onto, so $\bbG$ fails the path condition.  \\

 (2) $\Rightarrow$ (1): suppose that $\bbG$ fails the path condition. Then by Lemma \ref{lemma-onto-path} there exists an induced subpath $\bbP$ of $\bbG$ of length $n-1$, induced subpaths $\bbQ_j$ of $\bbG$ of length $n-2$ and a homomorphism from $\bbQ=\prod_j \bbQ_j$ onto $\bbP$. Let $a$ and $b$ denote the endpoints of $\bbP$, and let $\alpha,\beta \in \bbQ$ such that $f(\alpha)=a$ and $f(\beta)=b$. Let $a_j,b_j$ denote the coordinates of $\alpha$ and $\beta$ in $\bbQ_j$ respectively, and let $\bbR_j$ denote the induced subpath  of $\bbQ_j$ {\em from $a_j$ to $b_j$}. Notice that depending on which  of  $a_j$ or $b_j$ appears first in the path $\bbQ_j$, we have that $\bbQ_j \geq \bbR_j$ or  $\bbQ_j \geq \overline{\bbR_j}$. But because the set of paths $\bbQ_1,\dots,\bbQ_l$ is closed under taking duals, i.e. for every $i$ there exists $s$ such that $\overline{\bbQ_i}=\bbQ_s$, if the word $W$ satisfies $W \geq \bbQ_i$ for all $i$ then $W \geq \bbR_j$ for all $j$. Hence $\alpha W \beta$ in $\bbQ$, and thus $f(\alpha) W f(\beta)$ in $\bbP$, i.e. $W \geq \bbP$.

\end{proof}

\section{Cycles that fail the path condition: a characterisation}

 \begin{definition} We say that the cycle $\bbG$ is {\em almost symmetric} if $\bbG$ contains exactly one non-symmetric edge. \end{definition}

The main theorem of this section is the following: 

\begin{theorem} \label{theorem-char} Let $n \geq 3$. The $n$-cycle  $\bbG$ fails the path condition if and only if $\bbG$ is almost symmetric or 
 there exists a self-dual path $\bbS$ and a positive integer $k$ such that $\bbG = (\bbS**)^k\,\bbS+$. \end{theorem}
 
 We first make an easy but important observation. 
 
  \begin{lemma} \label{lemma-above}Let $W,A,B$ be words  such that $W \geq A + B$ and $W \geq A - B$. Then $W \geq A+-B$ or $W \geq A -+B$. \end{lemma}
 
 \begin{proof} Using Lemma \ref{lemma-utilities} (1)   it easy to see that $W \geq A+B$ implies we can write $W=U_1+U_2$ with $U_1 \geq A$ and $U_2 \geq B$; similarly we can write $W=V_1-V_2$ with $V_1 \geq A$ and $V_2 \geq B$. If $|V_1| < |U_1|$ then $W = V_1-Z+U_2$ for some word $Z$, and thus $W \geq A-+B$; otherwise $|U_1| < |V_1|$ and a similar argument shows that $W \geq A+-B$.

 \end{proof}
 
The next two lemmas prove the easy direction.

  \begin{lemma} \label{lemma-fail-1} Let $\bbG$ be an almost symmetric cycle of girth at least 3. Then $\bbG$ fails the path condition. \end{lemma} 
    
    \begin{proof} Let $n$ denote the girth of $\bbG$; let $\bbP = (*)^{n-1}$, let $\bbQ =  (*)^{n-3}\,-$ and let $\bbR =  (*)^{n-3}\,+$. Clearly $\bbP$ is an induced subpath of $\bbG$ of length $n-1$ and $\bbQ,\bbR$ are induced subpaths of $\bbG$ of length $n-2$. Let $W \geq \bbQ$ and $W \geq \bbR$. By Lemma \ref{lemma-above} we have that $W \geq (*)^{n-3}-+$ or  $W \geq (*)^{n-3}+-$; in any case $W \geq (*)^{n-1} = \bbP$. It follows by Theorem \ref{theorem-surj-path} that $\bbG$ fails the path condition. \end{proof}

 \begin{lemma}  \label{lemma-fail-2}  Let $\bbS$ be a self-dual path, let $k$ be a positive integer. Then the cycle $\bbG = (\bbS**)^k\, \bbS+$ fails the path condition. \end{lemma}
 
 \begin{proof}   Let $n$ denote the girth of $\bbG$. Clearly $\bbP =  (\bbS**)^{k}\,\bbS$ is an induced subpath of $\bbG$ of length $n-1$.  It is easy to see that the paths $\bbQ = \bbS+(\bbS**)^{k-1}\bbS$ and $\bbR = \bbS-(\bbS**)^{k-1}\bbS$ are induced subpaths of $\bbG$ of length $n-2$ (for $\bbR$, read the word for $\bbG$ ``backwards''). Let $W \geq \bbQ$ and $W \geq \bbR$. By Lemma \ref{lemma-above},  we  have that $W \geq  \bbS+-(\bbS**)^{k-1}\bbS$ or $ W \geq \bbS-+(\bbS**)^{k-1}\bbS$; in any case we have that $W \geq  \bbS**(\bbS**)^{k-1}\bbS$, i.e. $W \geq \bbP$. It follows by Theorem \ref{theorem-surj-path} that $\bbG$ fails the path condition. 
 
 \end{proof}
 
 \subsection{A necessary condition to fail the path condition}
 
 Our goal in this subsection is to prove the following technical result:

%\bnote{are A and B both non-empty ? Makes no difference for lemmas that follow. }
\begin{theorem} \label{theorem-david} Let $\bbG$ be a cycle that fails the path condition. Then there exists words $A$ and $B$ such that 
$\bbG = C(A**B+) = C(A+B**) = C(A-B**)$. \end{theorem}

Let $\bbG$ be an $n$-cycle and let $\bbP_i$ be its subpaths of length $n-2$. According to Theorem \ref{theorem-surj-path},  $\bbG$ satisfies the path condition if we can find, for each of its induced subpaths $\bbP$ of length $n-1$, a word $W \ngeq \bbP$ such that $W \geq \bbP_i$ for all $i$. It turns out that there is a uniform way of constructing, for each path $\bbP$, a word $W(\bbP)$ that will in most cases satisfy the desired condition; furthermore,  this construction will help us finding a characterisation for the cycles that fail the path condition. The next definition describes the construction.  

%%%%%%{The idea of infinite words is interesting, but does not work well formally (at least not in any way I can see at the moment). I will formulate definitions in terms of a %%%%%% large fixed power $N$ to be determined later. and even this does not work well ..... We'll see later if we can simplify.}

\begin{definition} Let $x,y \in \cA$, and let $N$ be a positive integer. 
\begin{enumerate} 
\item Let 
$$[x,y]= \begin{cases} max\{x,y\} & \mbox{ if $x$ and $y$ are comparable}, \\
yx & \mbox{ otherwise.}
\end{cases}$$

\item Let $p=(+)^N$ and $m = (-)^N$. 
\item Let $$\widehat{x}= \begin{cases} m & \mbox{ if $x$ is $+$}, \\
p & \mbox{  if $x$ is $-$}, \\
\epsilon &  \mbox{  if $x$ is $*$}.
\end{cases}$$
\item For any non-empty word $X=x_1\cdots x_r$, let 
$$W_N(X) = \widehat{x_1} [x_1,{x_2}]\widehat{x_2} \dots  \widehat{x_{r-1}}[x_{r-1},{x_r}]\widehat{x_r}.$$
\end{enumerate}
\end{definition}

\noindent {\bf Remark.} In the following, we will assume $N$ is a fixed integer, much larger than $n$ where $n$ is the girth of the cycle $\bbG$ in question, and we shall  omit the index $N$ as much as possible.

%\newpage

The following lemma follows directly from the definitions. 

\begin{lemma} \label{lemma-w} Let $x,y \in \cA$, and let $A,B,X$ be words. Then 
\begin{enumerate}
\item  $\widehat{x}\ngeq x$,
%\item $|\widehat{x}|\leq |x|$. 
\item  $[x,y]\geq x$ and  $[x,y]\geq y$ but $[x,y]\ngeq xy$.

%\item \bnote{check where this is used} If $X = AxyB$ then $W(X) = W(Ax)[x,y]W(yB)$. 
\end{enumerate}
\end{lemma}
\qed

\begin{lemma} \label{lemma-david-1-4} Let $x,y,z \in \cA$ with $x \neq *$, and let $a,b$ be integers such that $a+b \leq n$. 
\begin{enumerate}
\item $*^a\, \overline{x}\, *^b \leq \widehat{x}$;
\item  $*^a\, z \leq \widehat{x}\,[x,y]$;
\item $z \, *^a\ \leq [y,x]\,\widehat{x}$;
\item $yz  \not\leq [y,z]\,\widehat{z}$. 

\end{enumerate}
\end{lemma}

\begin{proof} (1) is immediate by the definitions. (2) follows from (1) if $z$ is $*$ or $\overline{x}$; otherwise $z=x$, and then notice that $*^a \leq \widehat{x}$ and $x \leq [x,y]$. (3) is identical to (2). An easy case-by-case analysis yields (4). 
\end{proof}

%%%{The 3rd  item of the next lemma is there as a pseudo-absorption. Not as nice, but formally correct. }

\begin{lemma}  \label{lemma-david-1-5-a} Let $X=x_1 \dots x_r$ be a non-empty word on $\cA$, and write $X = x_1A = Bx_r$. 
\begin{enumerate}
\item $W(X) \geq A$, $W(X) \geq B$, but $W(X) \not\geq X$; 
\item $W_{N+1}(X) \geq W_N(X)$;
\item  If $x_1 \in \{+,-\}$ then $W_{N+1}(X) \geq \overline{x_1}  \, W_N(X)$; if $x_r \in \{+,-\}$ then $W_{N+1}(X) \geq W_N(X) \, \overline{x_r}$. 

\end{enumerate}

\end{lemma}

\begin{proof} (1) The first two inequalities are immediate by Lemma \ref{lemma-w} (2). Now we prove that $W(X) \not\geq X$ by induction on $r$. If $r=1$ this is Lemma \ref{lemma-w} (1). Let $Y = X\, x_{r+1}$; then $W(Y) = W(X)\,[x_r,x_{r+1}]\,\widehat{x_{r+1}}$. Suppose for a contradiction that $W(Y) \geq Y$; by induction hypothesis, and using Lemma \ref{lemma-utilities} (1),  if $W(X) \not\geq X$ then we must have $[x_r,x_{r+1}]\,\widehat{x_{r+1}} \geq x_r \,x_{r+1}$, which contradicts Lemma \ref{lemma-david-1-4} (4). 

(2) is immediate. 

(3) Simply observe that $W_N(X)$ is of the form  $\widehat{x_1}\,Y\, \widehat{x_r}$ and that  $\overline{x}\widehat{x} = \overline{x}^{N+1}$ if $x \neq *$.  
\end{proof}

\begin{lemma}  \label{lemma-david-1-5-b} Let $X=x_1 \dots x_r$ be a non-empty word on $\cA$, let $A,B$ be any words on $\cA$ and let $z \in \cA$. 
\begin{enumerate}
\item $W(X*) = W(X)\,x_r$ and $W(*X) = x_1 \, W(X)$;
\item  $W(A\, *^k \, B) = W(A*)\, *^{k-1} \, W(*B)$ for all $k > 0$; 
\item $W_{N+1}(A*B) \geq W_N(AB)$; 
\item $W(A\,z\,B) \geq W(A*B)$;
\item If $A$ and $B$ are non-empty and $B \geq A$ then there exists $M \geq N$ such that \\$W_{M}(B)  \geq W_N(A) $. 
\end{enumerate}

\end{lemma}

\begin{proof} (1) is immediate. We prove (2) by induction on $k$. Suppose first that $k=1$: if $A$ or $B$ is empty the result is trivial. Otherwise, write $A=Ca$ and $B=bD$ and then 
$$W(A*B) =  W(A)[a,*]\, \widehat{*}\, [*,b]W(B) = W(A)\,ab\,W(B) = W(A*)W(*B)$$
the last equality following from (1).  Now suppose the result holds for $k$. Then using the result for $k=1$  we get
$$W(A*^{k+1}B) =  W(A*\,(*^{k}B))= W(A*)W(*^{k+1}B) $$
and using the  induction hypothesis we conclude that this is
$$ W(A*)W(*\,*^{k}B) = W(A*)W(**)*^{k-1}W(*B)=W(A*)*^{k-1}W(*B).$$

(3) By (2) we have that $W_{N+1}(A*B) = W_{N+1}(A*)W_{N+1}(*B)$. If $A$ or $B$ is empty the inequality is easy, so assume that $A=Ca$ and $B=bD$ for some $a,b \in \cA$. 
Then by (1) we get that $W_{N+1}(A*B) = W_{N+1}(A)\,ab\,W_{N+1}(B)$; and directly from the definition we have that $W_N(AB) = W_N(A)[a,b]W_N(B)$. If $a$ and $b$ are comparable, then the inequality follows from Lemma  \ref{lemma-david-1-5-a} (2). Otherwise, $W_N(AB) = W_N(A)[a,b]W_N(B) = W_N(A)\, ba \, W_N(B) = W_N(A)\, \overline{a}\overline{b} \, W_N(B)$ and the inequality now follows from Lemma \ref{lemma-david-1-5-a} (3). 

(4) If $A$ or $B$ is empty the result is easy, so assume that $A=Ca$ and $B=bD$ for some $a,b \in \cA$. Then $W(AzB) = W(A)[a,z]\,\widehat{z}\,[z,b]W(B) \geq W(A)\,ab\,W(B) = W(A*B)$ by (1), (2) and Lemma \ref{lemma-w} (2). 

(5) It is not hard to see, using Lemma \ref{lemma-utilities}, that $B \geq A$ precisely when $B$ contains a subword  from which $A$ can be obtained by replacing certain $+$ and $-$ by $*$. From this, it follows easily that there exists a word $D$ with the following properties: (i) $D$ is obtained from $B$ by replacing some $+$ and $-$ by $*$; (ii) $D$ contains $A$ as a subword; (iii) no symbol other than $*$ appears in $D$ outside the subword $A$. It follows from (4) that $W(B) \geq W(D)$. Now starting from $D$, repeatedly use (3), removing a $*$ at each step, until the subword $A$ is reached.

\end{proof}

\noindent{\bf Remark.} {\em In the following, given words $X,Y$ with $X \neq \epsilon$, we shall write $Y \leq W(X)$ to mean that there exists some $N$ such that $Y \leq W_N(X)$. Notice that this may lead to counter-intuitive statements such as $\widehat{x}\,W(X) \leq W(X)$ when $x$ is the first symbol of $X$ (see Lemma \ref{lemma-david-1-5-a} (3)). }

\begin{definition}  For a word $X$, let $X'$ denote the subword of $X$ obtained by deleting all the occurrences of $*$. \end{definition}

Notice that for any word $X$ we have that $\overline{X'} = {\overline{X}}'$. 

\begin{comment}
%\newpage

\bnote{Is this next decomposition  still used ?}

The next lemma is straightforward, and simply decomposes a word into substrings containing no occurrences of $*$. 

\begin{lemma} Let $X$ be a word  such that $X \neq X'$. Then there exists an integer $m > 0$, integers $k_i >0$ for $1 \leq i \leq m$,   words $X_i = X_i'$ for $1 \leq i \leq m+1$ where $X_i \neq \epsilon$ if $i \notin \{1,m+1\}$ such that   
$$X= X_1*^{k_1}X_2\dots  X_{m}*^{k_m}X_{m+1}.$$ 
We call this the {\em standard decomposition} of $X$. 
\end{lemma} \qed

\end{comment}

%\newpage

\begin{definition} Let $X, Y$ be non-empty words on $\cA$. A word $Z$  is a {\em shuffle} of $X$ and $Y$ if it has length $|X|+|Y|$ and contains $X$ and $Y$ as subwords. Denote by $\langle X,Y \rangle$ the set of all shuffles of $X$ and $Y$ that start with the first symbol of $X$ and end with the last symbol of $Y$. We denote by $\langle s, Y \rangle$ any set of the form $\langle *^a,Y \rangle$, and similarly for $\langle X,s \rangle$.  \end{definition}

\begin{lemma} \label{lemma-useful} Let $X$ be a non-empty word on $\{+,-\}$. 
\begin{enumerate}
\item If $Y \in \langle s, X \rangle$ then $Y \leq W(X*)$;
\item If $Y \in \langle X,s \rangle$ then $Y \leq W(*X)$.
\end{enumerate} \end{lemma}

\begin{proof} (1) Write $X=x_1\dots x_r$. By Lemma \ref{lemma-david-1-4} (2),  
$*^ax_i \leq \widehat{x_i}[x_i,x_{i+1}]$ for all $1 \leq i < r$ and $*^ax_r \leq \widehat{x_r}[x_r,*]$; concatenating yields the desired inequality. The proof of (2) is the same, using Lemma \ref{lemma-david-1-4} (3) instead. 
\end{proof}

\begin{prop} \label{prop-david-1-6}  Let  $X,Y$ be  words on $\cA$ such that $Y' \not\geq X'$. Then $Y \leq W(X')$ and in particular $Y \leq W(X)$.  \end{prop}

\begin{proof} Since $W(X') \leq W(X)$ by Lemma \ref{lemma-david-1-5-b} (5) it suffices  to prove that $Y \leq W(X')$. 
Write $X' = x_1\dots x_r$ and $Y' = y_1 \dots y_k$. Suppose first that $k < r$. By Lemma \ref{lemma-david-1-4} (2) we have for any $a \leq n$ that $*^ay_i \leq \widehat{x_i}[x_i,x_{i+1}]$ for all $i$  and $*^a \leq \widehat{x_{k+1}}$; since $Y$ is obtained from $Y'$ by inserting $*$'s, it follows that $Y \leq W(X')$. 

Now we prove the result for $k \geq r-1$ by induction on $k-(r-1)$. If $k-(r-1)=0$  the previous case holds, so now assume that $k \geq r$ and the result holds for smaller values. Since $X'$  is not a prefix of $Y'$, there exists words $G$, $D$, $E$ and $x,y \in \cA$ distinct such that $X' = G x D$ and $Y' = G y E$; notice that $\overline{x} = y$ and $\overline{y} = x$. 

(i) Suppose first that $G = \epsilon$: then $*^a\, \overline{x}\, *^b \leq \widehat{x}$ for any $a+b \leq n$ by Lemma \ref{lemma-david-1-4} (1), and by induction hypothesis any word $Z$ with $Z' = E$ will satisfy $Z \leq W(X')$. 
Hence $Y = *^a\, \overline{x}\, *^b \, Z \leq \widehat{x} W(X') \leq W(X')$. 

(ii) Now suppose that $G$ is non-empty. Since $Y'=GyE$, $Y$ is of the form  $K *^ay*^bL$ with $K \in \langle s,G \rangle$ and $L'=E$. 
First we have that  $K \leq W(G*)$ by Lemma \ref{lemma-useful}; secondly $*^a \overline{x} *^b \leq \widehat{x}$ for all $a+b \leq n$, by Lemma  \ref{lemma-david-1-4} (1), and  thirdly, by induction hypothesis, we have that $L \leq W(xD)$ for any word $L$ such that $L' = E$.  Since $\widehat{x}W(xD) \leq W(xD)$, concatenating the above we get
$Y \leq W(G*) \widehat{x}W(xD) \leq W(G*)W(xD).$ By Lemma \ref{lemma-david-1-5-a} (1), if $z$ is the last symbol in G then 
$ W(G*)W(xD) = W(G)zW(xD) \leq W(G)[z,x]W(xD) = W(GxD) = W(X')$ and we are done. 

\end{proof}

The next result states that words $X$ and $Y$ such that $X'=Y'$  are comparable precisely when their stars align.  

\begin{lemma} \label{lemma-aligned} Let $X,Y$ be words on $\cA$ such that $X' = Y' = c_1\dots c_r$. Write $X = *^{x_0} c_1 \dots *^{x_{r-1}}c_r*^{x_r}$ and 
$Y = *^{y_0} c_1 \dots *^{y_{r-1}}c_r*^{y_r}$. Then $X \geq Y$ if and only if $x_i \geq y_i$ for all $0 \leq i \leq r$. \end{lemma}

\begin{proof} Clearly if $x_i \geq y_i$ for all $i$ we have $X \geq Y$. For the other direction suppose that $X \geq Y$: by Lemma \ref{lemma-utilities} (1) and the fact that an edge-preserving map cannot send a star to a non-star, it is easy to see that the monotone map from $X$ onto $Y$ restricted to $X'$ is an isomorphism onto $Y'$, and hence must preserve the intervals of $*$'s; the result follows immediately.  
\end{proof}

\begin{prop} \label{prop-david-1-7}
 Let  $X$ be a word on $\cA$ with $X'$ is non-empty. If $Y$ is a word on $\cA$ such that $Y' = X'$  but $Y \not\geq X$ then  $Y \leq W(X)$.
\end{prop}

\begin{proof}  By Lemma \ref{lemma-aligned}, if $X'=Y'$ but $Y \not\geq X$ then the following holds: \\

\noindent{\bf Claim.} {\em There exists words $X_g, Y_g, X_d, Y_d$ and integers $a > b$ such that 
$X = X_g *^a X_d$ and $Y=Y_g *^b Y_d$ where 
 $X_g,Y_g \in \langle s,X'_g \rangle$ or are both empty, and
 $X_d,Y_d  \in \langle X'_d,s \rangle$ or are both empty. 
}

%\noindent{\em Proof of Claim.} We use induction on $|X'|$. If $|X'|=1$ then $X = *^{u_1} x *^{u_2}$ and $Y = *^{v_1} x *^{v_2}$ for some $x \in \{+,-\}$; since $Y \not\geq X$ %we must have $u_1 > v_1$ or $u_2 > v_2$ and the claim holds. Now suppose that $|X'| \geq 2$. Then $X=X_1 *^{u_1} x\, *^{u_2}$ and $Y=Y_1 *^{v_1} x \,*^{v_2}$ for some %words $X_1, Y_1$ with $X'_1 = Y'_1 \neq \epsilon$ and $x \in \{+,-\}$; furthermore we can assume neither of $X_1,Y_1$ ends with a $*$. If $u_1 > v_1$ or $u_2 > v_2$ the c%laim holds; otherwise $Y \not\geq X$ implies $Y_1 \not\geq X_1$ and we can apply the induction hypothesis; the claim follows easily. 

Since $X' \neq \epsilon$, at least one of $X'_g$ or $X'_d$ is non-empty: suppose without loss of generality that $X'_d \neq \epsilon$. Then by Lemma \ref{lemma-useful} we have that $Y_d \leq W(*X'_d)$. Suppose first that $X'_g$ is empty. Then 
$Y = *^b\, Y_d \leq *^{a-1}W(*X'_d) \leq *^{a-1}W(*X_d) = W(*)*^{a-1}W(*X_d) = W(*^aX_d)=W(X)$ by Lemma \ref{lemma-david-1-5-b} (2) and (5). On the other hand if $X'_g$ is non-empty, then  by Lemma \ref{lemma-useful} $Y_g \leq W(X'_g*)$, so 
$Y = Y_g\,*^b\, Y_d \leq W(X'_g*)\,*^{a-1}W(*X'_d) = W(X'_g\,*^a \, X'_d) \leq W(X_g *^a X_d) = W(X)$ by Lemma \ref{lemma-david-1-5-b} (2) and (5).

 \end{proof}
 
 Let $\bbG$ be an $n$-cycle that fails the path condition. By Theorem \ref{theorem-surj-path}  there exists a subpath $\bbP$ of $\bbG$ of length $n-1$ and a subpath $\bbQ$ of $\bbG$ of length $n-2$ such that $\bbQ \not\leq W(\bbP)$. The next result uses the previous two propositions to give a sharper description of such cycles.

 \begin{prop}\label{prop-david-1-8} Let $\bbG = C(\bbP\,e) = C(\bbQ\,xy)$ be an $n$-cycle, where $\bbP, \bbQ$ are words on $\cA$ and $x,y,e \in \cA$. If $\bbQ \not\leq W(\bbP)$ then 
 (i)  $xy = **$, (ii) $e \in \{+,-\}$, and (iii) there exist words $A,B$ such that $\bbP'=AB$ and either $\bbQ' = AeB$ or $\bbQ'=A\overline{e}B$.

 \end{prop}
 
 \begin{proof} By Propositions \ref{prop-david-1-6} and \ref{prop-david-1-7} we must have that $\bbQ' > \bbP'$; in particular, $\bbQ$ must contain the edge $e$ outside $\bbP$, and $e \neq *$;  (iii) also follows.  Now $|\bbQ'| \leq |\bbG'| = |\bbP'|+1 \leq |\bbQ'|$ implies $xy = **$. 
 %Suppose first that $\bbQ$ is in the same direction as $\bbP$ along $\bbG$. Write $\bbG = UxyVe$ where $\bbP=UxyV$ and $\bbQ = VeU$. Then $\bbP'=U'V'$ and $                          %\bbQ'=V'eU'$; in particular $h(\bbQ')=h(\bbP')+h(e)$. Since $\bbQ' > \bbP'$, there exist words $A,B$ and $z \in \cA$ such that $\bbP'=AB$ and $\bbQ'=AzB$; then  $h(\bbQ') = %h(A)+h(B)+h(z) = h(\bbP')+h(z)$ and so $h(z)=h(e)$, i.e. $z=e$. If $\bbQ$ runs counter to $\bbP$, then \bnote{fix this: see handwritten proof} 
 \end{proof}
  
  For the remainder of this subsection, $\bbG$ is an $n$-cycle failing the path condition, as witnessed by a subpath $\bbP$ of $\bbG$ of length $n-1$; 
  $\bbQ$ is a subpath of $\bbG$ of length $n-2$ such that $\bbQ \not\leq W(\bbP)$.
  
  \begin{lemma} \label{lemma-decompo} There exist words $A,B,C,D,T,U$, $z \in \{+,-\}$ and an integer $k \geq 0$ such that $\bbP = ATB$ and $\bbQ=CUD$ with
  \begin{enumerate}
  \item $A'=C'$, $B'=D'$, $T'=z^k$, $U'=z^{k+1}$,
  \item if $A' \neq \epsilon$ then its last symbol is $\overline{z}$; if $B' \neq \epsilon$ then its first symbol is $\overline{z}$;
  \item $A,C$ do not end with a $*$; $B,D$ do not start with a $*$. 
  \end{enumerate} 
  \end{lemma}
  
  \begin{proof} By Proposition \ref{prop-david-1-8} we can find words $X,Y$ such that $\bbP'=XY$ and $\bbQ'=XzY$. Choose $X$ the shortest possible in such a decomposition; if it is non-empty then clearly it must end with $\overline{z}$. Now let $W$ be the longest suffix of $Y$ that starts with $\overline{z}$ (if no such word exits then $W=\epsilon$). Then clearly $\bbP'=Xz^kW$ and $\bbQ'=Xz^{k+1}W$ where $k \geq 0$. To complete the proof it suffices to insert the stars in the appropriate places. 
  \end{proof}
 
  \begin{prop} \label{prop-david-1-9} In the above decomposition,  $C \geq A$ and $D \geq B$. 
  \end{prop}

\begin{proof} We prove the result by contradiction: we show that, if $\bbP$ and $\bbQ$ admit the decomposition of Lemma \ref{lemma-decompo} and $C \not\geq A$ or $D \not\geq B$ then $\bbQ \leq W(\bbP)$; by symmetry it suffices to consider the case  $C \not\geq A$. In that case $A \neq \epsilon$ so $A=A_1a$ with $a = \overline{z}$ by Lemma \ref{lemma-decompo} (2) and (3).
\begin{enumerate}
\item  $C \leq W(A)$ by Proposition \ref{prop-david-1-7}.
\item $U \leq \widehat{a}$ by repeated applications of Lemma \ref{lemma-david-1-4} (1) with $x = a = \overline{z}$.
\item $D \leq W(*B)$; indeed, if $D = \epsilon$ the result is trivial. Otherwise by Lemma \ref{lemma-decompo} (3) we have that $D \in \langle D',s\rangle =  \langle B',s\rangle$ so $D \leq W(*B') \leq W(*B)$ by Lemma \ref{lemma-useful} (2) and Lemma \ref{lemma-david-1-5-b} (5). 
\end{enumerate}

Now if we concatenate the above we get that $\bbQ = CUD \leq W(A)\widehat{a}\,W(*B) \leq W(A)W(*B)$ by Lemma \ref{lemma-david-1-5-a} (3). If $B = \epsilon$ the result follows easily; otherwise by Lemma \ref{lemma-decompo} (2) and (3) we have $B = aB_1$; by  Lemma \ref{lemma-david-1-5-b} (1) and (5) we get

$$ W(A)W(*B) = W(A)a\,W(B)   =W(A)[a,a]W(B) 
 = W(AB) 
 \leq W(ATB)  = W(P).
$$
\end{proof}

Knowing that $\bbP$ contains exactly two more $*$'s than $\bbQ$, and that $C \geq A$ and $D \geq B$ by Proposition \ref{prop-david-1-9},  a simple counting argument using Lemma \ref{lemma-aligned} shows that $T$ contains at least two more $*$'s than $U$, and in particular $|T|  \geq |U|+1$ so it is not empty.) 

\begin{lemma} \label{lemma-U-and-T-a} $U \not\leq W(T)$. \end{lemma}

\begin{proof} By the argument used in (3) of the proof of Proposition \ref{prop-david-1-9} we have that $C \leq W(A*)$ and $D \leq W(*B)$. Let us now suppose that neither $A$ nor $B$ is empty (otherwise the argument is quite similar): let $A=A_1a$ and $B=bB_1$, and let $u$ and $v$ denote the first and last symbol of $T$. If $U \leq W(T)$, we get by Lemma \ref{lemma-david-1-5-b} (1) that
\begin{align*}
 \bbQ = CUD \leq W(A*)W(T)W(*B) &= W(A)aW(T)bW(B)\\
 & \leq W(A)[a,u]W(T)[v,b]W(B) \\
 & = W(ATB) = W(\bbP),
 \end{align*}
a contradiction. 

 \end{proof}
 
 We can now refine Lemma \ref{lemma-decompo}:
 
 \begin{lemma}  \label{lemma-decompo-b} There exist words $\underline{A},\underline{B},\underline{C},\underline{D}$ on $\cA$, $z \in \{+,-\}$ and an integer $t \geq 2$ such that 
 
 \begin{enumerate}
 
 \item $\bbP = \underline{A}*^t\underline{B}$ and $\bbQ = \underline{C}z\underline{D}$;
 \item   $\underline{A}'=\underline{C}'$, $\underline{B}'=\underline{D}'$;
% \item $\underline{A} \leq \underline{C}$ and $\underline{B} \leq \underline{D}$; \bnote{ this is used to prove (6) but nothing else ?}
% \item $\underline{A},\underline{C}$ do not end with a $*$; $\underline{B},\underline{D}$ do not start with a $*$;
% \item $t=a+b+2$;
 \item $\underline{C} \leq W(\underline{A}*)$ and $\underline{D} \leq W(*\underline{B})$;
 \item there exist integers $a,b$ and words $_AX,{_BX},X_C,X_D$ such that $\underline{A}={_AX}*^a$, $\underline{C}={_CX}*^a$, $\underline{B}=*^bX_B$, $\underline{D}=*^bX_D$ where ${_AX},{_CX}$ do not end with $*$ and $X_B,X_D$ do not start with $*$(in other words, $\underline{A}$ and $\underline{C}$ end with the same number of $*$'s and $\underline{B}$ and $\underline{D}$ start with the same number of $*$'s);
 \item If $t=2$, then $\underline{A}=\underline{C}$ and $\underline{B}=\underline{D}$. 
 
 \end{enumerate}

 \end{lemma}

 \begin{proof} Let $U =  *^{u_0} z \dots *^{u_{k}}z*^{u_{k+1}}$ and $T =   *^{t_0} z \dots *^{t_{k-1}}z*^{t_{k}}$. Recall that $\sum_{i=0}^k t_i > \sum_{i=0}^{k+1} u_i$ by the comments just prior to Lemma \ref{lemma-U-and-T-a}. 
 
 \noindent{\bf Claim 1.} There exist indices $l \leq r$ such that $u_l < t_l$ and $u_{r+1} < t_r$. 
 
 \noindent{\em Proof of Claim 1.}  We use induction on $k$.
 If $k=0$ the result is immediate. Now suppose that $k \geq 1$ and the result holds for smaller values. Since $\sum_{i=0}^k t_i > \sum_{i=0}^{k} u_i$ and $\sum_{i=0}^k t_i > \sum_{i=1}^{k+1} u_i$, there exist indices $l$ and $r$ such that $u_l < t_l$ and $u_{r+1} < t_r$; it remains to show that there are such indices such that $l \leq r$. If $u_0 < t_0$ or $u_{k+1} < t_k$ we are done, so assume that $u_0 \geq t_0$ and $u_{k+1} \geq t_k$. It follows that  $\sum_{i=1}^{k-1} t_i > \sum_{i=1}^{k} u_i$ and we can use the induction hypothesis to conclude. 
 
  \noindent{\bf Claim 2.}  For any indices satisfying the conditions of Claim 1,  $l = r$. 
 
 \noindent{\em Proof of Claim 2.}  We prove this by contradiction: suppose that $u_l < t_l$ and $u_{r+1} < t_r$ and $l < r$: we shall deduce that $U \leq W(T)$, contradicting Lemma \ref{lemma-U-and-T-a}. Since $l < r$, we can write $T = T_l *^{t_l} T_c *^{t_r} T_r$ and $U = U_l *^{u_l} U_c *^{u_r} U_r$ where for some $m \geq 0$ we have $T'_c = z^m$, $U'_c = z^{m+1}$, $T'_l=U'_l$ and $T'_r=U'_r$; furthermore  $U_l$ does not end with $*$, $U_r$ doesn't start with $*$, and $U_c$ starts and ends with $z$. 
 \begin{enumerate}
 \item $U_l \leq W(T_l*)$: same argument as used in (3) of the proof of Proposition \ref{prop-david-1-9};
 \item $U_r \leq W(*T_r)$: same comment.
 \item $U_c \leq W(*T_c*)$: indeed, for some $a_i$ we have
 $$U_c = (z*^{a_1})(z*^{a_2}) \cdots(z*^{a_m})z \leq (z \widehat{z})^mz = z (\widehat{z}z)^{m-1}\widehat{z}z = W(*z^m*)=W(*T'_c*)\leq W(*T_c*)$$
 using Lemma \ref{lemma-david-1-5-b} (5) for the last inequality. 
 \end{enumerate}
 If we concatenate the above, and using the fact that $*^{u_l} \leq *^{t_l-1}$ and  $*^{u_{r+1}} \leq *^{t_r-1}$ we get 
 \begin{align*} U = U_l *^{u_l} U_c *^{u_r} U_r  &\leq    W(T_l*)*^{t_l-1}W(*T_c*)  *^{t_r-1} W(*T_r)  \\
 &=           W(T_l*^{t_l}T_c *^{t_r}T_r)      = W(T)
  \end{align*}
 using Lemma \ref{lemma-david-1-5-b} (2) twice. \\
 
 By Claims 1 and 2 there is an index $j$ such that $u_j < t_j$ and $u_{j+1} < t_j$. Then we can write $U = U_l *^{u_j} z *^{u_{j+1}} U_r$ and $T = T_l *^{t_j} T_r$. By Claim 2 we have $u_{i} \geq t_i$ for all $i \leq j-1$ and $u_{i+1} \geq t_i$ for all $i \geq j+1$; since $T$ contains at least two stars more than $U$ it follows that we can write  $t_j=u_j+u_{j+1}+t$ for some $t \geq 2$.  Clearly $T'_l = U'_l$ and $T'_r = U'_r$.
 Let  $\underline{A}=AT_l*^{u_j}$, $\underline{B}=*^{u_{j+1}}T_rB$, $\underline{C}=CU_l*^{u_j}$ and $\underline{D}=*^{u_{j+1}}U_rD$.  It is immediate that  $\underline{A}'=\underline{C}'$ and $\underline{B}'=\underline{D}'$.  Notice that (4) is satisfied by construction and Lemma \ref{lemma-decompo}.

  Since $U_l$ does not end with a star, we use the argument in (3) of the proof of Proposition \ref{prop-david-1-9} to get that $CU_l \leq W(AT_l*)$ and thus 
  $$\underline{C} =CU_l*^{u_j} \leq W(AT_l*)*^{u_j} = W(AT_l*^{u_j+1})=W(\underline{A}*)$$ using Lemma \ref{lemma-david-1-5-b} (2). The argument for the other inequality is identical.
  
  Finally we prove (5): suppose $t=2$. One verifies easily, using Proposition \ref{prop-david-1-9} and Lemma \ref{lemma-aligned} that $\underline{A} \leq \underline{C}$ and $\underline{B} \leq \underline{D}$. For any word $X$ let $\sigma(X) = |X|-|X'|$, i.e. the number of stars in $X$.  Since $\bbG = C(\bbP e)=C(\bbQ{*}{*})$ we have that $\sigma(\bbP) = \sigma(\bbQ)+2$. By Lemma \ref{lemma-aligned}, $\underline{A} \leq \underline{C}$ and $\underline{B} \leq \underline{D}$ implies $\sigma(\underline{A}) \leq \sigma(\underline{C})$ and  $\sigma(\underline{B}) \leq \sigma(\underline{D})$.  Then $\sigma(\bbP) = \sigma(\underline{A}) + \sigma(\underline{B}) + 2 \leq \sigma(\underline{C})+\sigma(\underline{D})+2 = s(\bbQ)+2$ so $\sigma(\underline{A}) = \sigma(\underline{C})$ and  $\sigma(\underline{B}) = \sigma(\underline{D})$; applying Lemma \ref{lemma-aligned} again we get $\underline{A}=\underline{C}$ and $\underline{B}=\underline{D}$. 
 \end{proof}
  
We require the following extension of Lemma \ref{lemma-david-1-5-a} (1): not only do we have $W(X) \not\geq X$, but this holds if we replace any number of $*$'s in $W(X)$ by any choice of $+$'s and $-$'s. 

\begin{lemma}\label{lemma-extension} Let $X$ be a non-empty word, let $Y$ be a word with a subword $Z$ on $\{+,-\}$ such that replacing every symbol of $Z$ by $*$ in $Y$ yields the word $W(X)$. Then $Y \geq W(X)$ and $Y \not\geq X$.   \end{lemma}

\begin{proof} The first inequality is immediate. To prove the second statement we use induction on the length of $X$. Since every $*$ of $W(X)$ can only appear as $[*,* ]$, if $X$ does not contain consecutive $*$'s then $W(X)$ contains no $*$'s and the result is trivial (in particular if $|X|=1$). If $Z = \epsilon$ then $Y=W(X)$ and the result holds by Lemma \ref{lemma-david-1-5-a} (1), so we suppose that $Z \neq \epsilon$. 
So write $X=A**B$ where the middle $**$ corresponds to a symbol in the subword $Z$. By Lemma \ref{lemma-david-1-5-b} (2) we have that $W(A**B)=W(A*)*W(*B)$.  Thus $Y=Y_1zY_2$ where $z \in \{+,-\}$, and $Y_1$ (respectively $Y_2$)  is obtained from $A*$ (respectively $*B$) by replacing various $*$'s by any choice of $+$'s and $-$'s. By induction hypothesis, we have that $Y_1 \not\geq A*$ and $Y_2 \not\geq *B$. By Lemma \ref{lemma-utilities} (1), if $Y_1zY_2 \geq A**B$ and $Y_1 \not\geq A*$ then we must have $zY_2 \geq **B$ which implies $Y_2 \geq *B$ by Lemma \ref{lemma-utilities} (5), a contradiction. 
\end{proof}

 \begin{definition} \label{def-position} 
 Let $\bbQ$ be any subpath of $\bbG$ of length $n-2$ such that $\bbQ \not\leq W(\bbP)$. Then Lemma \ref{lemma-decompo-b} gives a representation of $\bbP$ and $\bbQ$ that identifies a substring of $*$'s in the word $W(\bbP)$ as follows: we have that $W(\bbP)=W(\underline{A}*^t \underline{B}) = W(\underline{A}*)*^{t-1}W(*\underline{B})$ by  Lemma \ref{lemma-david-1-5-b} (2). 
 Let $p(\bbQ)=\{i: |W(\underline{A}*)|+1 \leq i \leq |W(\underline{A}*)|+(t-1)\}$. Let $s(\bbQ)=z$ (in the representation given in Lemma \ref{lemma-decompo-b}). 
 \end{definition}

\begin{prop}\label{prop-final-decompo} Let $\bbG$ be an $n$-cycle failing the path condition, as witnessed by the subpath $\bbP$ of length $n-1$. Then there exist subpaths $\bbQ, \bbR$ of $\bbG$ of length $n-2$, words $A,B,C,D,E,F$ on $\cA$ and an integer $t \geq 2$ such that 
\begin{enumerate}
\item $\bbG=C(\bbP e)=C(\bbQ**)=C(\bbR**)$ for some $e \in \{+,-\}$, 
\item $p(\bbQ) \subseteq p(\bbR)$,
\item $\bbP=A*^2 B = C*^tD$,
\item $\bbQ = A+B$ and $\bbR = E-F$,
%\item $C \leq E$ and $D \leq F$, 
\item $C$ and $E$ end with the same number of $*$'s; \\ $D$ and $F$ start with the same number of $*$'s; 
\item $A'=C'=E'$ and $B'=D'=F'$. 
\end{enumerate}
\end{prop}

\begin{proof} \mbox{}\\

\noindent{\bf Claim 1.} {\em There exist $\bbQ,\bbR \not\leq W(\bbP)$ such that $s(\bbQ) \neq s(\bbR)$, $p(\bbQ) \subseteq p(\bbR)$ and $|p(\bbQ)|=1$. }\\

\noindent{\em Proof of Claim 1.}  It follows immediately from Lemma \ref{lemma-decompo-b} and Definition \ref{def-position} that, if $Y$ is obtained from $W(\bbP)$ by replacing a $*$ whose position is in $p(\bbQ)$ by $s(\bbQ)$, then $\bbQ \leq Y$. Suppose for a contradiction that our claim is false. We define a word $Y$ by replacing $*$'s in $W(\bbP)$ as follows: for every $\bbQ \not\leq W(\bbP)$  such that $|p(\bbQ)|=1$, change the $*$ in that position by $s(\bbQ)$ (this is well-defined, by hypothesis). The remaining subword of $*$'s of $W(\bbP)$ is replaced by an alternation  $(+-)^r$. We claim that $Y \geq \bbQ$ for every $\bbQ \not\leq W(\bbP)$. Indeed, this is clear if $p(\bbQ)$ contains $p(\bbR)$ with $|p(\bbR)|=1$ since in that case $s(\bbQ)=s(\bbR)$. Otherwise, $p(\bbQ)$ is an interval of size at least 2 and hence must contain positions of $*$'s that we replaced by $+$ and $-$. By Lemma \ref{lemma-extension}, this contradicts the failure of the path condition witnessed by $\bbP$. \\

We get (1) from Claim 1 and Proposition \ref{prop-david-1-8}. Since $|p(\bbQ)|=1$, it means that $t=2$ in the representation for $\bbP$ and $\bbQ$; then all the statements of our proposition follow from Lemma \ref{lemma-decompo-b} except $A'=C'$ and $B'=D'$. We have that 
$$W(\bbP)=W(A*)*W(*B) = W(C*)*^{t-1} W(*D)  $$ and since $p(\bbQ) \subseteq p(\bbR)$ we have that $W(A*)=W(C*)*^s$  and $W(*B)=*^rW(*D)$ for some $s,r \geq 0$. The result will then follow from the next two claims:\\

\noindent{\bf Claim 2.} {\em Let $Y$ be a proper, non-empty prefix  of $X$. Then $|W(Y*)| < |W(X*)|$.   }\\

\noindent{\em Proof of Claim 2.} This is straightforward and left to the reader. \\

\noindent{\bf Claim 3.} {\em Let $X$ and $Y$ be words such that one is a prefix of the other, and let $s \geq 0$. If $W(X*)=W(Y*)*^s$ then $X'=Y'$. If $X$ and $Y$ are such that one is a suffix of the other and $W(*X)=*^sW(*Y)$ then $X'=Y'$. }\\

\noindent{\em Proof of Claim 3.} We prove the first statement, the argument for the other is identical. We may suppose neither $X$ nor $Y$ is empty; we use induction on $s$. If $s=0$ then $X=Y$ by the previous claim. Suppose the result holds for some fixed $s \geq 0$. If $W(X*)=W(Y*)*^{s+1}$ then $|W(Y*)| < |W(X*)|$ so $Y$ is a proper prefix of $X$ by the last claim. We also have that $W(X*)$ ends with a $*$ which means that  $X=X_1*$ for some $X_1$. Then  $W(Y*)*^{s+1}= W(X_1**)=W(X_1*)*$ by Lemma \ref{lemma-david-1-5-b} (2). Thus $W(X_1*)=W(Y*)*^s$ where $Y$ is a prefix of $X_1$ so by induction hypothesis we have that $Y'=X_1'=X'$ and we are done. 

 \end{proof}
 
  \begin{definition}  Let $W$ be a word on $\{+,-,*\}$. Define the {\em height} $h(W)$ of $W$ as the number of $+$ in $W$ minus the number of $-$ in $W$. 
\end{definition}

Clearly $h(UV) = h(U)+h(V)$, $h(\overline{W})=-h(W)$ and $h(U')=h(U)$  for any words $U, V, W$. In particular, if a word $W$ is self-dual then $h(W)=0$.

\begin{lemma} \label{lemma-rotation} If $W, Z$ are words such that  $C(W)=C(Z)$ then $h(W) = \pm h(Z)$;\\ furthermore, if $h(W) \neq 0$ then the following are equivalent: 
\begin{enumerate}
\item $h(W) \neq h(\overline{Z})$;
\item $h(W)=h(Z)$;
\item there exist words $U$ and $V$ such that 
$W= UV$ and $Z=VU$. \end{enumerate}

 \end{lemma}
 
 \begin{proof}   The first statement is clear. Now suppose $h(Z)\neq 0$. Then (3) implies (2) and (2) implies (1) are immediate. Now suppose that $h(W) \neq h(\overline{Z})$. Since $Z$ represents the same cycle as $W$, either $Z$ or $\overline{Z}$ is a ``rotation'' of $W$, i.e. of the form $VU$ where $W=UV$. But since $h(\overline{Z})  \neq h(W)$, $\overline{Z}$ cannot be obtained this way, so the result follows.  \end{proof}

\begin{comment}
\begin{lemma}\label{lemma-1-fix}  Let $A$ and $B$ be words such that $C(B**A+) = C(\overline{A}**\overline{B}+)$. Then  $B**A = \overline{A}**\overline{B}$.\end{lemma}

\begin{proof} Since $h(A**B+) = h(A)+h(B)+1 \neq h(A)+h(B)-1=h(B**A-)$, by Lemma \ref{lemma-rotation} we must have $h(A)+h(B)+1=h(A**B+)=h(\overline{A}**\overline{B}+) = -h(A)-h(B)+1$, so $h(A)+h(B)=0$; in particular the cycle has height 1 and we can apply the second part of Lemma \ref{lemma-rotation}:  there exist words $U,V$ such that 
\begin{enumerate}
\item[(i)] $B**A+=UV$,
\item[(ii)] $\overline{A}**\overline{B}+=VU$. 
\end{enumerate}
Suppose for a contradiction that $|U||V| \geq 1$; then both $U$ and $V$ end with a $+$, and it follows from (i) that we must have $|U| \leq |B|$ or $|V| \leq |A|$. The argument is identical in both cases, so assume $|U| \leq |B|$: by (i) there exists words $B_1$, $B_2$ such that $U = B_1+$ and $B=B_1+B_2$. Then by (ii) we get 
$$VU = \overline{A}**\overline{B_2}-\overline{B_1}+$$
and comparing the lengths of the suffixes we conclude that $V$ must end with $-$, a contradiction. 
\end{proof}
Here is a more general, better version.
\end{comment}

\begin{lemma}\label{lemma-1-fix}  Let $A,B,E,F$  be words such that $A'=E'$ and $B'=F'$. \\
If $C(A+B**) = C(E-F**)$, then  $B**A = \overline{E}**\,\overline{F}$.\end{lemma}

\begin{proof} We have that $h(A)=h(E)$ and $h(B)=h(F)$. Since $h(A+B**) = h(A)+h(B)+1 \neq h(A)+h(B)-1=h(E-F**)$, by Lemma \ref{lemma-rotation} we must have $h(A)+h(B)+1=-(h(A)+h(B)-1)$  so $h(A)+h(B)=0$; in particular the cycle has non-zero height and we can apply the second part of Lemma \ref{lemma-rotation}:  there exist words $U,V$ such that 
$B**A+=UV$,
and $\overline{E}**\,\overline{F}+=VU$. 
Suppose for a contradiction that $|U||V| \geq 1$; then both $U'$ and $V'$ are non-empty and end with a $+$. Applying $'$ to both sides of these equalities we get 
\begin{enumerate}
\item[(i)] $B'A'+=U'V'$,
\item[(ii)] $\overline{A'}\overline{B'}+=V'U'$. 
\end{enumerate}

and it follows from (i) that we must have $|U'| \leq |B'|$ or $|V'| \leq |A'|$. The argument is identical in both cases, so assume $|U'| \leq |B'|$: by (i) there exists words $B_1$, $B_2$ such that $U' = B_1+$ and $B'=B_1+B_2$. Then by (ii) we get 
$$V'U' = \overline{A'}\overline{B_2}-\overline{B_1}+$$
and comparing the lengths of the suffixes we conclude that $V'$ must end with $-$, a contradiction. 

\end{proof}

 We can now give the proof of the main result of this section:\\

\noindent{\em Proof of Theorem  \ref{theorem-david}:} 

Consider the decomposition of paths $\bbP$, $\bbQ$ and $\bbR$ given by Proposition \ref{prop-final-decompo}: we prove that we have $t=2$. Indeed, there exist words $A_1,B_1,C_1,D_1,E_1,F_1$ and integers $a,b,e,f \geq 0$ such that $A=A_1*^a$, $E=E_1*^e$, $C=C_1*^e$, $B=*^bB_1$, $F=*^fF_1$, $D=*^fD_1$, and $A_1,C_1,E_1$ do not end with  $*$ and $B_1,D_1,F_1$ do not start with $*$.  Then $A**B= A_1*^{a+2+b}B_1 = C_1*^{e+t+f}D_1$, and since $A'_1 = A'=C'=C'_1$ we conclude that $a+2+b=e+t+f$. 

On the other hand, we have that $C(A+B**) = C(E-F**)$ and $A'=E'$, $B'=F'$ so Lemma \ref{lemma-1-fix} gives us that $B**A= \overline{E}**\overline{F}$. Thus $*^bB_1**A_1*^a= *^e\overline{E_1}**\overline{F_1}*^f$, so $b=e$ and $a=f$. Then $t=2$. 

 It follows from Lemma \ref{lemma-decompo-b} (5) that $C=E$ and $D=F$; since $A**B=C**D$ and $p(\bbQ) \subseteq  p(\bbR)$ by Proposition \ref{prop-final-decompo} (2), we have that $|W(A*)|=|W(C*)|$ so $A=C$ by Claim 2 in Proposition \ref{prop-final-decompo}. 
 We conclude that $A=C=E$ and $B=D=F$. 
 
 Now suppose that $e=+$ in the decomposition: then  $\bbG = C(A+B**)=C(A-B**) = C(A**B+)$ and we are done. If $e=-$, then $\bbG = C(A+B**)=C(A-B**) = C(A**B-)$; 
 reading the words ``backwards'', we get the desired decomposition using $\overline{B}$ and $\overline{A}$ instead of $A$ and $B$.

 \qed

\subsection{Proof of Theorem \ref{theorem-char}} We now proceed to finish the proof of Theorem \ref{theorem-char}. We require a few more lemmas. Yes. Sorry. 

%%%%%%
%%%%%

\begin{lemma}\label{lemma-2-fix} Let $A$ and $B$ be words such that $C(A-B**) = C(A**B+)$. Then  $\overline{A}=A$ and $\overline{B}=B$.
 \end{lemma}
 
 \begin{proof} Since $h(A-B**) = h(A)+h(B)-1 \neq  h(A)+h(B)+1= h(A**B+)$, by Lemma \ref{lemma-rotation} we must have $h(A)+h(B)+1=h(A**B+)=h(\overline{A}**\overline{B}+) = -h(A)-h(B)+1$, so $h(A)+h(B)=0$; in particular the cycle has non-zero height and we can apply the second part of Lemma \ref{lemma-rotation}:
  there exist words $U,V$ such that 
\begin{enumerate}
\item[(i)]  $\overline{A}**\overline{B}+=UV$,
\item[(ii)]$A**B+=VU$. 
\end{enumerate}
 Suppose for a contradiction that $|U||V| \geq 1$; then both $U$ and $V$ end with a $+$, and it follows from (ii) that we must have $|V| \leq |A|$ or $|U| \leq |B|$.
 Suppose first that $|U| \leq |B|$: it follows from (ii) that there exist words $B_1$, $B_2$ and $b \in \{+,-,*\}$ such that $B=B_1bB_2$ and $U=B_2+$. Then $V=A**B_1b$ so concatenating we get
 $$UV =  B_2+A**B_1\, b$$
  But from (i) we get 
 $$ UV =  \overline{A}**\overline{B_2}\, \overline{b}\, \overline{B_1}+ $$
 and comparing the suffixes of length $|B_1|+2$ we obtain $b=+$ and $\overline{b}=*$, a contradiction. 
 Now suppose that $|V| \leq |A|$: then by (ii) there exist words $A_1$ and $A_2$ such that $A=A_1+A_2$ and $V = A_1+$; then $U = A_2 ** B +$ and concatenating we get that 
 $$UV = A_2 ** B+A_1+$$
   But from (i) we get 
 $$ UV =  \overline{A_2}-\overline{A_1}**\overline{B}+$$
 and comparing the prefixes of length $|A_2|+1$ we obtain $*= -$ another contradiction. 
   \end{proof}

 %%%%%%
%%%%%

\begin{lemma} \label{lemma-symmetry} Let $A,B$ be words on $\{+,-,* \}$ such that  $A* * \, B = B * *\, A$. Then there exist integers $k,l \geq 0$ such that  one of the following holds:
\begin{enumerate}
\item $A = (*)^k$ and $B = (*)^l$, or \item there exists a word $S$ such that $A = (S **)^k\, S$ and $B = (S **)^l\, S$; furthermore, if $\overline{A}=A$ or $\overline{B}=B$ then $\overline{S}=S$.
\end{enumerate}
 \end{lemma}

\begin{proof} We prove the result by induction on $|A|+|B|$. If $|A|=|B|$ then clearly $A=B$ and we may take $k=l=0$ and $S = A = B$. Thus by symmetry of the statement we may now suppose without loss of generality that $|A|<|B|$. 

\begin{enumerate} 
\item[(i)] Suppose first that $|B|=|A|+1$. We get that $B=A \, * = *\, A$. It follows easily that $A = (*)^k$ and  $B = (*)^{k+1}$ for some $k$. 

\item[(ii)] Suppose that $|B| = |A|+2$. We get that $B=A \, ** = **\, A$. It is easy to see that in this case $A = (*)^k$ and  $B = (*)^{k+2}$ for some $k$. 

(Notice that, in particular, the above proves the statement for $|A|+|B| \leq 2$.) 

\item[(iii)] Now suppose that $|B| > |A|+2$; then considering the prefixes in $A* * \, B = B * *\, A$ we can write $B = A\, **\, W$ for some non-empty word $W$. Then we get 
$A* * \, A\, **\, W = A\, **\, W * *\, A$ and considering suffixes we conclude that $A\, **\, W =  W * *\, A$. Since $|W|<|B|$, by induction hypothesis we have that either both $A$ and $W$ are of the form $(*)^q$ and thus $B$ also, or otherwise there exists some word $S$ such that $A =  (S **)^k\, S$ and $W = (S **)^l\, S$. Then $B = A\, **\, W =  (S **)^k\, S \, **\, (S **)^l\, S =  (S **)^{k+l+1}\, S$ and we are done. 

\end{enumerate}
To prove the second statement in (2), simply compare prefixes in  $(S **)^q\, S$ and  $\overline{(S **)^q\, S} = \overline{S} \, \overline{(S **)^q}$. 

\end{proof}

\noindent{\em Proof of Theorem \ref{theorem-char}:} One direction has been proved in Lemmas \ref{lemma-fail-1} and \ref{lemma-fail-2}. Let $\bbG$ be a cycle that fails the path condition; we must show it has the desired form. By Theorem \ref{theorem-david}, there exist words $A$ and $B$ such that $\bbG = C(A**B+) = C(A+B**) = C(A-B**)$.
By  Lemma \ref{lemma-1-fix} (with $E=A$ and $F=B$) we get that $B**A=\overline{A}**\overline{B}$; by Lemma \ref{lemma-2-fix} we get that  $\overline{A}=A$ and $\overline{B}=B$; together these give $A**B = B**A$  so we conclude using Lemma \ref{lemma-symmetry}.  \qed

%%%%%%%%
%%%%%%%%
%%%%%%%%%%%
%%%%%%%%%

%\newpage
\section{Cycles without the path condition are S\l upecki}
\label{section-five}

%subsection{A sufficient condition} \label{prelim-theorem}

In a companion paper \cite{larose-pullas-preprint}, we prove a criterion that guarantees that under certain technical conditions, a reflexive digraph is S\l upecki; in particular it applies to any reflexive digraph that triangulates a sphere. Since cycles clearly triangulate a circle, we can apply this criterion:

\begin{theorem}  \label{theorem-embedding} \cite{larose-pullas-preprint} Let $\bbG$ be an $n$-cycle with $n \geq 4$. If for every $p \geq 2$ and every onto polymorphism $f:\bbG^p \rightarrow \bbG$  there exists an embedding $e:\bbG \hookrightarrow \bbG^p$ such that the restriction of $f$ to $e(\bbG)$ is onto then  $\bbG$ is S\l upecki.  \qed
  \end{theorem}

%\subsection{Cycles without the path condition are S\l upecki} 

We now apply Theorem \ref{theorem-embedding} to show that the remaining cycles, those that do not satisfy the path condition, are S\l upecki: once an arbitrary onto polymorphism is fixed, the copy of $\bbG$ in $\bbG^p$ on which $f$ is onto is built by glueing together various paths.

% {  \begin{theorem} \label{theorem-almost} Let $\bbG$ be an almost symmetric cycle of girth at least 4. Then $\bbG$ is S\l upecki. \end{theorem}
    
   % \begin{proof}
  %  Let $\bbG$ have vertex set $\{0,\dots,n-1\}$, and let $(n-1,0)$ be the only non-symmetric edge of $\bbG$. \\
  
%\noindent{\bf Claim 0.} {\em $ u (*)^{n-1} v$ for all $u,v \in \bbG^k$. }\\
% The claim is immediate since clearly $ u (*)^{n-1} v$ for all $u,v \in \bbG$.\\

%\noindent{\bf Claim 1.} {\em $ u +(-)^{n-3} v$ for all $u,v \in \bbG^k$. }\\
%It suffices to show this holds in $\bbG$: fix $u,v \in G$; if there exists a symmetric path of length at most $n-2$ between $u$ and $v$ we are done; otherwise, $\{u,v\}=\{0,n-1\}$ %and then the claim is immediate. \\

%Let $f:\bbG^p \rightarrow \bbG$ be a surjective polymorphism, and let $f(u) = 0$ and $f(v)=n-1$; by Claim 1, there is a $+(-)^{n-3}$-path from $v$ to $u$; hence there exists $u'$ %such that $v \rightarrow u'$ and $u' (-)^{n-3} u$. 
%Then we have that $n-1=f(v)\rightarrow f(u')$ so $f(u') \in \{0,n-1,n-2\}$. However,  no $(+)^{n-3}$-path starting at 0 can reach $n-1$ nor $n-2$, hence $f(u')=0$. By Claim 0, %there is a symmetric path of length at most $n-1$ from $v$ to $u'$; applying $f$ to this path, we see that it must map onto the symmetric path between  0 and $n-1$;  in %particular we have an embedding of $\bbG$ in $\bbG^p$ on which $f$ is onto. We can now apply Theorem \ref{theorem-embedding} and conclude that $\bbG$ is S\l upecki. 

   % \end{proof}}.     

\begin{lemma} \label{lemma-automorphisms} If $\bbG=C(W)$ is an $n$-cycle with $h(W)=1$ then $Aut(\bbG)$ is trivial.  \end{lemma}

\begin{proof} The automorphisms of $\bbG$ are certainly automorphisms of the undirected cycle and thus are either rotations or reflections. Choose an orientation and denote by $\tau_k$ the rotation sending $0$ to $k$ ``clockwise''. 
Suppose for a contradiction that $ Aut(\bbG)$ is non-trivial; if it contains some $\tau_k$ with $1 < k < n$ then because $\tau_n$ is also an automorphism, applying B\'ezout's lemma, $Aut(\bbG)$ contains the rotation $\tau_d$ where $d$ is the gcd of $k$ and $n$; this implies that the word $W$ can be written as $V^m$ where $dm=n$. But $1=h(W) = mh(V)$ implies $m=1$ so $d=n$, a contradiction. Otherwise, $Aut(\bbG)$ contains  a non-trivial reflection $\sigma $; obviously $\sigma$ cannot preserve an oriented edge, and hence it is easy to see that, if $U$ is a  substring of $W$ of maximum length such that $\sigma(U) \cap U = \emptyset$, then $h(W) = h(U) - h(U) = 0$. 
\end{proof}

\begin{lemma}  \label{lemma-alternating}  Let $s$ be a positive integer. 

\begin{enumerate} 

\item  $(-+)^s- \geq W$ for every word $W$ on $ \cA$ of length at most $s+1$ except $(+)^{s+1}$;
% $(+-)^s+ \geq W$ for every word $W$ on $ \cA$ of length at most $s+1$ except $(-)^{s+1}$;

\item  $(+-)^{s} \geq W$ for every word $W$   on $\cA$ of length at most $s$, and every $W$ of length $s+1$ containing at least one $*$ except those of the form $W = (-)^r\,*\,(+)^l$.

\end{enumerate}

\end{lemma}

\begin{proof}  We shall use the obvious fact that any word $W$ of length at most $k$ is below both $(+-)^k$ and $(-+)^k$. \\

(1) Clearly $(-+)^s- \not\geq (+)^{s+1}$. For the other direction, we may assume $W$ has length $s+1$ and is non-constant. We can then insert at most $s-1$ symbols from $\{+,-\}$ in $W$ to obtain an alternating word of length at most $2s$; if it has length at most $2s-1$ then we can, if need be, add $-$ at one or both ends; in the case it has length $2s$, since it is alternating it starts or ends with a $-$ and thus we can add $-$ at the other end and we are done.  \\

(2) It is easy to see that  $(+-)^{s} \not\geq  (-)^r\,*\,(+)^l$ if $r+l=s$. For the other statement, we can assume $W$ has length $s+1$ and is not of the form $W = (-)^r\,*\,(+)^l$.  Write $W=W_1*W_2$ where $|W_1|=a$ and $|W_1|=b$ with $a+b=s$. Without loss of generality (the other case is dealt with using the dual of statement (1)), $W_2$ is non-empty and not equal to $(+)^b$. By (1) we have that 
$W_1*W_2 \leq (+-)^a * (-+)^{b-1}- =  (+-)^a (*-) (+-)^{b-1} \leq (+-)^s$.

\end{proof}

Observe that cycles that fail the path condition satisfy a simple necessary condition: by Theorem \ref{theorem-char}, clearly every such cycle $\bbG$ of girth at least 4 is of the form $\bbG=C(\bbP+)$ for some self-dual path $\bbP$ of length at least 3  which contains at least two $*$'s.

\begin{theorem} \label{theorem-last} Let $\bbP$ be a self-dual path such that $|P| \geq 3$ and $|\bbP|-|\bbP'| \geq 2$.  Then the cycle $\bbG = C(\bbP \, +)$
 is S\l upecki.\end{theorem}

%% this originally said P contains consecutive stars, but looking at the proof it seems all se need is that P contains at least 2 stars. 

\begin{proof}  First notice that since $\bbP$ is self-dual, $h(\bbP)=0$ so $Aut(\bbG)$ is trivial by Lemma  \ref{lemma-automorphisms}. 
Let $(a,b)$ be the arc of $\bbG$ corresponding to the $+$ (see Figure \ref{cycle-even-odd}). 
We split the proof in two cases, according to the parity of $|\bbP|$, although the structure of both proofs is the same. 

%\begin{figure}[htb]
%\begin{center}
%\includegraphics[scale=0.5]{cycle-even-odd.pdf}
 %\caption{The cycle $\bbG = C(\bbP \, +)$ (clockwise, from $b$), with $|\bbP|$ even (left) and  $|\bbP|$ odd (right).} \label{cycle-even-odd}
% \end{center}
%\end{figure}

\begin{figure}[htb]
\begin{center}
  \begin{minipage}[b]{0.4\textwidth}
    \includegraphics[width=\textwidth]{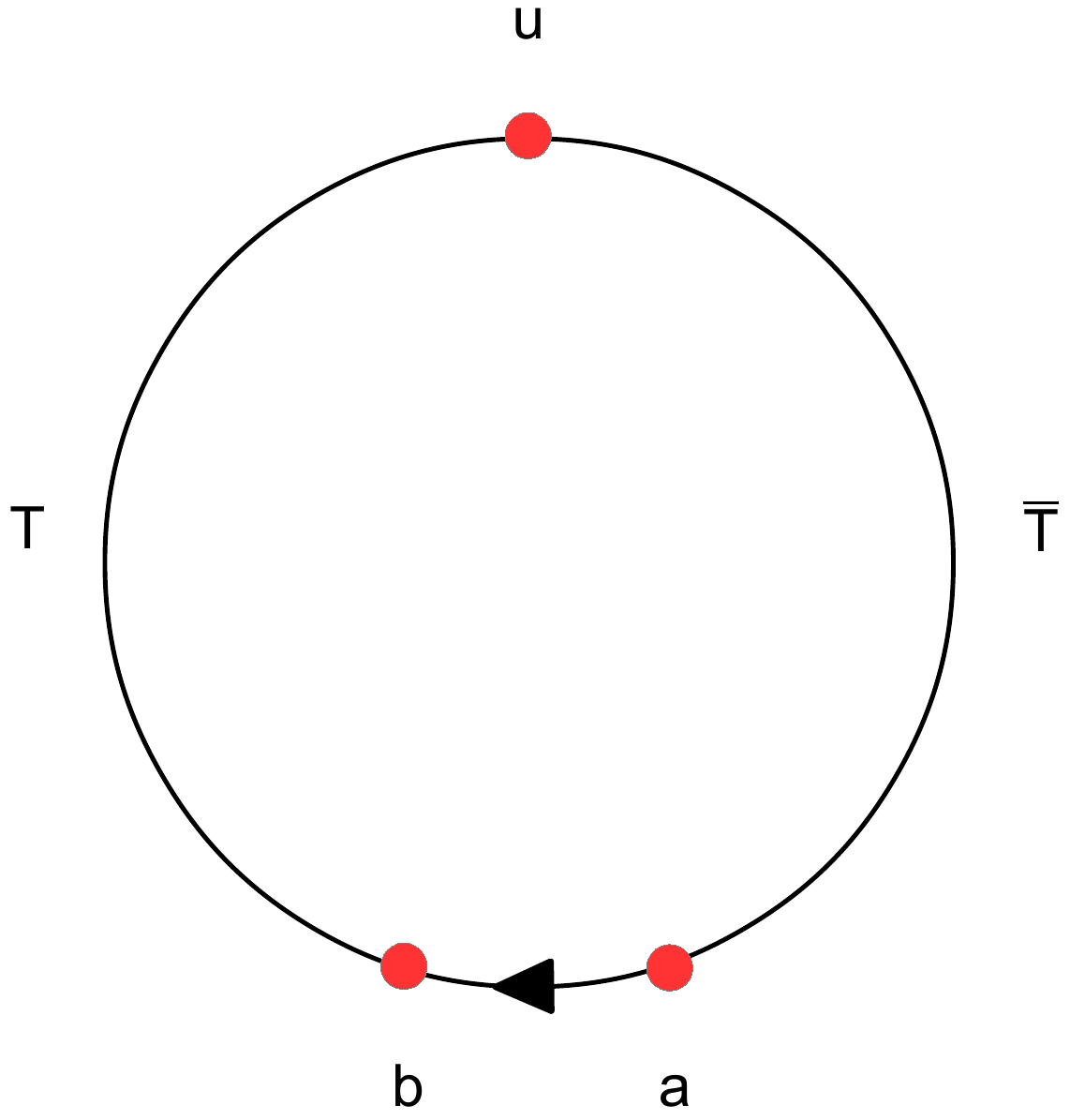}
  \end{minipage}
  \hspace{2cm}
  \begin{minipage}[b]{0.4\textwidth}
    \includegraphics[width=\textwidth]{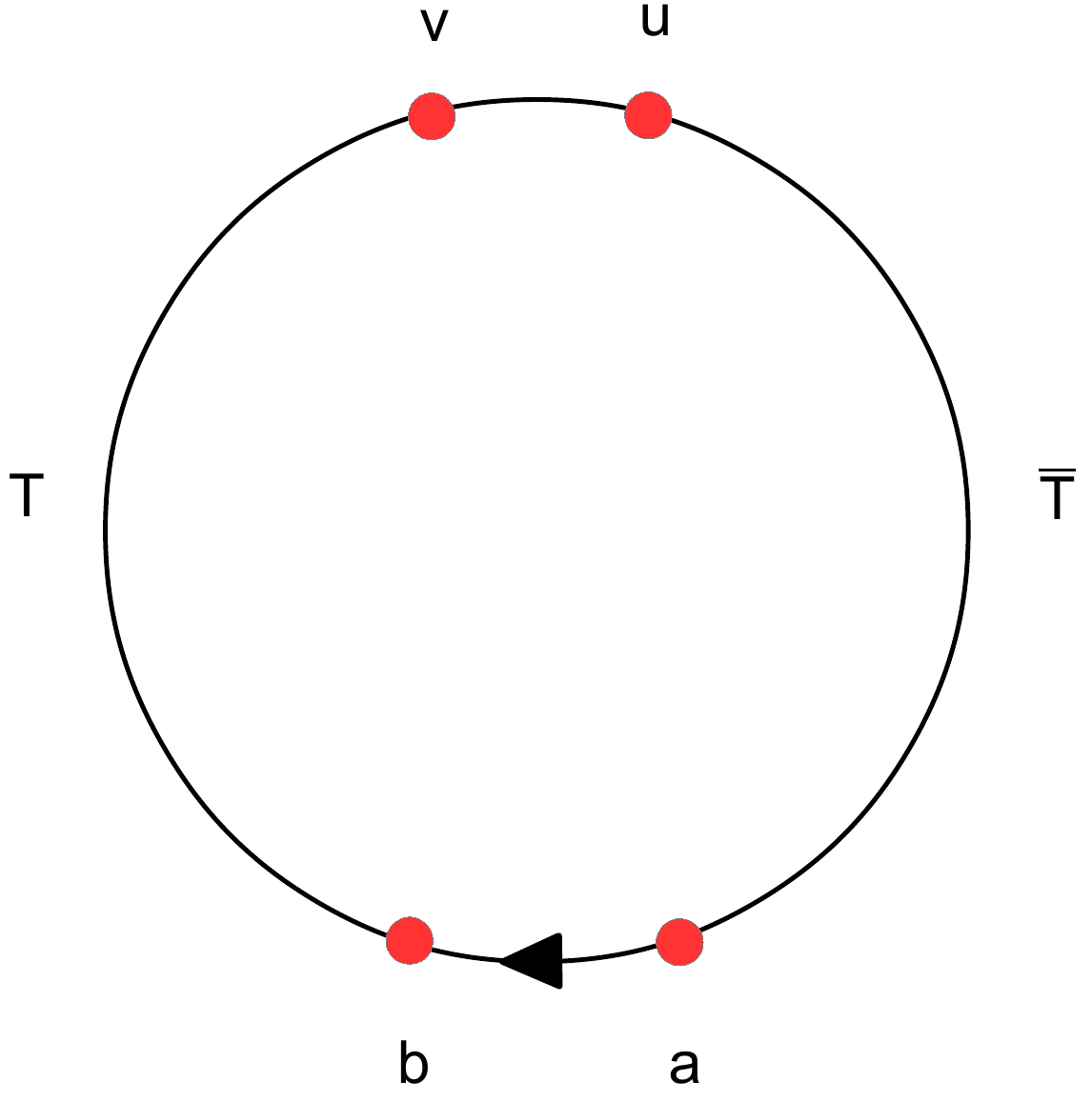}
  
  \end{minipage}
  \caption{The cycle $\bbG = C(\bbP \, +)$ (clockwise, from $b$), with $|\bbP|$ even (left) and  $|\bbP|$ odd (right).} \label{cycle-even-odd}
 \end{center}
\end{figure}

\begin{enumerate}

\item {\bf $|\bbP|$ is odd.} Then $\bbP=\bbT \, * \, \overline{\bbT}$ for some path $\bbT$; let $\{u,v\}$ be the vertices incident to the symmetric edge corresponding to the middle $*$, see  Figure \ref{cycle-even-odd}.  Let $t$ denote the length of $\bbT$.

 \noindent{\bf Claim 0o.} {\em Let $Z = (+-)^{t}$. Then $x \, Z \, y$ for all $x,y \in G$.}\\

 \noindent {\em Proof of Claim 0o.}  The cycle has girth $2t+2$, hence the (usual graph) distance on the cycle from $x$ to $y$ is at most $t+1$. 
Apply Lemma \ref{lemma-alternating} (2): if the usual graph distance between $x$ and $y$ is $t$ or less we are done; otherwise $x$ and $y$ are antipodal; since $\bbP$ contains at least three $*$,  there is a word from $x$ to $y$ of length $t+1$ which contains at least two $*$ and we are done in this case also. \\

 \noindent{\bf Claim 1o.}  {\em For all $x \in G$, either $x \bbT u$ or $x \bbT v$. }\\
 
  \noindent {\em Proof of Claim 1o.}  Immediate. \\

 \noindent{\bf Claim 2o.}  {\em Let $\{c,d\}$ be a symmetric edge of $\bbG$ such that $x \bbT c$ or $x \bbT d$ for all $x \in G$. Then $\{c,d\} = \{u,v\}$.  }\\
 
  \noindent {\em Proof of Claim 2o.}  Let $c'$ and $d'$ denote the respective antipodes of $c$ and $d$ on the cycle; by hypothesis the subpath $\bbY$ of $\bbG$ of length $t$ that connects $d'$ and $c$ satisfies $\bbY \leq \bbT$ and similarly for the subpath $\bbZ$  of $\bbG$ of length $t$ that connects $c'$ and $d$. Hence  the map $h:G \rightarrow G$ that sends the path $b \bbT v$ to the path $d' \bbY c$ and the path $a \bbT u$ to the path $c' \bbZ d$ is either a bijective endomorphism, or its composition with the reflection exchanging $c'$ and $d'$ is. In any case, we obtain an automorphism of $\bbG$ mapping the edge $\{u,v\}$ to the edge $\{c,d\}$. Since $Aut(\bbG)$ is trivial  we are done. \\
  
  \noindent{\bf Claim 3o.}  {\em For any  onto polymorphism $f$ from $\bbG^p$ onto $\bbG$ we have \\$f(\{u,v\}^p) \subseteq \{u,v\}$. }\\
  
    \noindent {\em Proof of Claim 3o.} Since $\{u,v\}^p$ is a clique,  $f(\{u,v\}^p) \subseteq \{c,d\}$ for some symmetric edge $\{c,d\}$ of $\bbG$.  Given any $x \in G$, there exists $\gamma \in G^p$ such that $f(\gamma) = x$. By Claim 1o there exists some $\delta \in \{u,v\}^p$ such that $\gamma \bbT \delta$ and so $x \bbT f(\delta) \in \{c,d\}$. We conclude by applying Claim 2o.    \\ 
  
   Let $f$ be a surjective polymorphism from $\bbG^p$ onto $\bbG$. Let $\alpha,\beta \in \bbG^p$ such that $f(\alpha)=a$ and $f(\beta)=b$. By Claim 0o we have $ \alpha ((-+)^{t}) \beta$, and since the longest substring of $\bbG$ from $b$ to $a$ has length $2t+1 > 2t$, the image of $ \alpha ((-+)^{t}) \beta$ under $f$ must use the arc $(a,b)$; hence without loss of generality we have that $(\alpha,\beta)$ is an arc in $\bbG^p$. By Claim 1o there exist $\gamma, \delta \in \{u,v\}^p$ such that $\alpha \bbT \gamma$ and $\beta \bbT  \delta$, and hence 
$a \bbT f(\gamma)$ and $b \bbT f(\delta)$; by Claim 3o and a simple distance argument it follows  that $f(\gamma) = u$ and $f(\delta)=v$. Because $\{u,v\}^p$ is a clique $(\gamma,\delta)$ is a symmetric edge, and the paths  $\alpha \bbT \gamma$ and $\beta \bbT \delta$ must both be isomorphic to  $\bbT$ since their endpoints are mapped to $a$ and $u$, and to $b$ and $v$ respectively, and so are mapped to the unique path of the correct length between those. Hence we have found a copy of $\bbG$ in $\bbG^p$ on which $f$ is onto, and we conclude using Theorem \ref{theorem-embedding}. \\

\item {\bf $|\bbP|$ is even.} Then $\bbP=\bbT \overline{\bbT}$ for some path $\bbT$; let $u$ denote the middle vertex of the path $\bbP$, see  Figure \ref{cycle-even-odd}. 
Let $t$ denote the length of $\bbT$. Notice that $t \geq 2$ since  $|P| \geq 3$. 

Let $\bbP_t = (+)^{t-1}\,** (-)^{t-1}$. 

 \noindent{\bf Claim 0e.} {Let $Z = (-+)^{t-1}-$. If $\bbG \neq C(\bbP_t\,+)$ then $x \, Z \, y$ for all $x,y \in G$; in any case $x \, (-+)^t\, y$ for all $x,y \in G$. }

 \noindent {\em Proof of Claim 0e.} Notice first that $(-+)^t \geq Z$, and so if we can prove $x \, Z \, y$  for a pair $x,y$ then the second statement follows for the same pair. The cycle has girth $2t+1$, hence the (usual graph) distance on the cycle from $x$ to $y$ is at most $t$.  Applying  Lemma \ref{lemma-alternating} (1) with $s=t-1$, we can handle all cases except if there is a path from $x$ to $y$ that equals $(+)^t$; since the arcs incident to $u$ are involutes of one another this path cannot contain both of them, and since both $\bbT$ and $\overline{\bbT}$ contain a $*$ this path in fact cannot use either of these arcs.  Similarly, the arcs incident to the arc $(a,b)$ are involutes of one another, so the path cannot contain them both; it follows that unless $x=a$, $y$ is the neighbour of $u$ at distance $t$ from $a$ and $\bbT = (+)^{t-1}\,*$, we do have that $x Z y$.  For this lone case, we observe that the remaining part of the cycle (which is  $C(\bbP_t\,+)$) gives $x Y y$ where $Y = (+)^{t-1}\,**$, and clearly $(-+)^t = (-+)^{t-1}(-+) \geq Y$. \\

 \noindent{\bf Claim 1e.}  {\em $x \bbT u$ for all $x \in G$. }\\
 
  \noindent {\em Proof of Claim 1e.}  Immediate. \\

 \noindent{\bf Claim 2e.}  {\em If $x \bbT v$ for all $x \in G$ then $v=u$. }\\
 
  \noindent {\em Proof of Claim 2e.}  Let $c$ and $d$ denote the two vertices at distance $t$ from $v$ on the cycle; by hypothesis the subpath $\bbY$ of $\bbG$ of length $t$ that connects $c$ and $v$ satisfies $\bbY \leq \bbT$ and similarly for the subpath $\bbZ$  of $\bbG$ of length $t$ that connects $d$ and $v$. Hence  the map $h:G \rightarrow G$ that sends the path $b \bbT u$ to the path $c \bbY v$ and the path $a \bbT u$ to the path $d \bbZ v$ is either a bijective endomorphism, or its composition with the reflection that fixes $v$ and exchanges $c$ and $d$ is. In any case, we obtain an automorphism of $\bbG$ mapping  $u$ to $v$; since $Aut(\bbG)$ is trivial we are done.  \\
  
  \noindent{\bf Claim 3e.}  {\em For any polymorphism $f$ from $\bbG^p$ onto $\bbG$ we have \\$f(u,u,\dots,u)= u$. }\\
  
    \noindent {\em Proof of Claim 3e.} This follows from Claims 1e and 2e and the fact that $f$ is onto. \\
    
     Let $f$ be a surjective polymorphism from $\bbG^p$ onto $\bbG$. Let $\alpha,\beta \in \bbG^p$ such that $f(\alpha)=a$ and $f(\beta)=b$. Suppose first that $\bbP \neq \bbP_t$. By Claim 0e we have $ \alpha ((-+)^{t-1}-) \beta$, and since the longest substring of $\bbG$ from $b$ to $a$ has length $2t > 2t-1$, the image of $ \alpha ((-+)^{t-1}-) \beta$ under $f$ must use the arc $(a,b)$;  in the case $\bbP = \bbP_t$, we have that $ \alpha ((-+)^{t}) \beta$; since the longest substring of $\bbG$ from $a$ to $b$ starts  with a $+$, a similar length argument shows that  the image of $ \alpha ((-+)^{t}) \beta$ under $f$ must use the arc $(a,b)$; hence without loss of generality we have that $(\alpha,\beta)$ is an arc in $\bbG^p$, in all cases. Let $\gamma = (u,u,\dots,u)$. By Claim 1e $\alpha \bbT \gamma$ and $\beta \bbT \gamma$, and by Claim 3e the paths  $\alpha \bbT \gamma$ and $\beta \bbT \gamma$ must both be isomorphic to  $\bbT$ since their endpoints are mapped to $a$ and $u$, and to $b$ and $u$ respectively, and so are mapped to the unique path of the correct length between those. Hence we have found a copy of $\bbG$ in $\bbG^p$ on which $f$ is onto, and we conclude using Theorem \ref{theorem-embedding}. \\

\end{enumerate}
\end{proof}

 \bibliographystyle{plain}
\bibliography{../ben-12-01-2016-new}

\end{document}